\documentclass[a4paper,11pt,reqno]{amsart}

\usepackage{amsmath}
\usepackage{amsfonts}
\usepackage{amssymb}
\usepackage{amsthm}
\usepackage[shortlabels]{enumitem}
\usepackage{graphicx}
\usepackage{color}
\usepackage{dsfont}
\usepackage[latin1]{inputenc}
\usepackage{MnSymbol}
\usepackage{enumitem}
\usepackage{extpfeil}
\usepackage{mathtools}
\usepackage{url}
\usepackage[margin=1.05in]{geometry}
\usepackage{fancyhdr}
\usepackage[all,cmtip]{xy}

\numberwithin{equation}{section}

\title{K\"ahler groups and subdirect products of surface groups}
\author{Claudio Llosa Isenrich}
\address{Max Planck Institute for Mathematics, Vivatsgasse 7, 53111 Bonn, Germany}
\email{llosa@mpim-bonn.mpg.de}
\urladdr{https://guests.mpim-bonn.mpg.de/llosa/}


\begin{document}

\newcommand{\AAA}{{\mathds A}}
\newcommand{\CC}{{\mathds C}}
\newcommand{\PP}{{\mathbf P}}
\newcommand{\QQ}{{\mathds Q}}
\newcommand{\RR}{{\mathds R}}
\newcommand{\NN}{{\mathds N}}
\newcommand{\ZZ}{{\mathds Z}}
\newcommand{\del}{{\partial}}
\newcommand{\one}{{\mathds {1}}}
\newcommand{\ord}{{\mathcal {O}}}
\newcommand{\ii}{{\mathds {i}}}
\newcommand{\vol}{{\mathrm {vol}}}
\newcommand{\eps}{{\epsilon}}
\def\Mod{{\rm{Mod}}}
\def\C{{\mathds C}}
\def\D{{\rm D}}
\def\S{\Sigma}
\def\F{{\mathds F}}
\def\FF{\mathcal F}
\def\aut{{\rm{Aut}}}
\def\inn{{\rm{Inn}}}
\def\out{{\rm{Out}}}
\def\isom{{\rm{Isom}}}
\def\mcg{{\rm{MCG}}}
\def\ker{{\rm{ker}}}
\def\im{{\rm{im}}}
\def\dim{{\rm{dim}}}
\def\G{\Gamma}
\def\a{\alpha}
\def\g{\gamma}
\def\L{\Lambda}
\def\Z{{\mathds{Z}}}
\def\H{{\mathds{H}}}
\def\n{{\bf N}}
\newcommand{\undll}{{\underline{l}}}
\newcommand{\mm}{{\underline{m}}}
\newcommand{\nn}{{\underline{n}}}

\theoremstyle{plain}
\newtheorem{theorem}{Theorem}[section]
\newtheorem{acknowledgement}[theorem]{Acknowledgement}
\newtheorem{claim}[theorem]{Claim}
\newtheorem{conjecture}[theorem]{Conjecture}
\newtheorem{corollary}[theorem]{Corollary}
\newtheorem{exercise}[theorem]{Exercise}
\newtheorem{lemma}[theorem]{Lemma}
\newtheorem{proposition}[theorem]{Proposition}
\newtheorem{question}{Question}
\newtheorem*{thmIntro1}{Theorem \ref{thmNewCoab}}
\newtheorem*{question2}{Question \ref{qnIntroKGsFinProps}}
\newtheorem*{question3}{Question \ref{qnIntroSuciuFinPres}}
\newtheorem*{question*}{Question}
\newtheorem{addendum}[theorem]{Addendum}

\newtheorem{keytheorem}{Theorem}
\renewcommand{\thekeytheorem}{\Alph{keytheorem}}

\theoremstyle{definition}
\newtheorem{remark}[theorem]{Remark}
\newtheorem*{acknowledgements*}{Acknowledgements}
\newtheorem{example}[theorem]{Example}
\newtheorem{definition}[theorem]{Definition}
\newtheorem*{notation*}{Notation}
\newtheorem*{convention*}{Convention}

\renewcommand{\proofname}{Proof}

\begin{abstract}    
We present a construction that produces infinite classes of K\"ahler groups that arise as fundamental groups of fibres of maps to higher dimensional tori. Following the work of Delzant and Gromov, there is great interest in knowing which subgroups of direct products of surface groups are K\"ahler.  We apply our construction to obtain new classes of irreducible, coabelian K\"ahler subgroups of direct products of $r$ surface groups. These cover the full range of possible finiteness properties of irreducible subgroups of direct products of $r$ surface groups: For any $r\geq 3$ and $2\leq k \leq r-1$, our classes of subgroups contain K\"ahler groups that have a classifying space with finite $k$-skeleton while not having a classifying space with finitely many $(k+1)$-cells. 

We also address the converse question of finding constraints on K\"ahler subdirect products of surface groups and, more generally, on homomorphisms from K\"ahler groups to direct products of surface groups. We show that if a K\"ahler subdirect product of $r$ surface groups admits a classifying space with finite $k$-skeleton for $k>\frac{r}{2}$, then it is virtually the kernel of an epimorphism from a direct product of surface groups onto a free abelian group of even rank.
\end{abstract}

\maketitle

\section{Introduction}
 A \textit{K\"ahler group} is a group that can be realised as the fundamental group of a closed K\"ahler manifold. The problem of which finitely presented groups are K\"ahler was first addressed by Serre in the 1950s \cite{Ser-58, JohRee-87} and has driven a field of very active research since. While numerous strong constraints have been proved and examples of K\"ahler groups with a variety of different properties have been constructed, the question remains wide open. For a general background on K\"ahler groups see \cite{ABCKT-95}, for a more recent overview see \cite{Bur-10}. 

 While a general answer seems out of reach for the moment, it is fruitful to consider Serre's problem in the context of more specific classes of groups. For instance, it has been shown that if the fundamental group of a compact 3-manifold without boundary is K\"ahler then it is finite \cite{DimSuc-09} (see also \cite{BisMjSes-12} and \cite{Kot-12-II}) and that a K\"ahler group with non-trivial first L2-Betti number is commensurable to a \textit{surface group} (i.e. the fundamental group of a closed Riemann surface) \cite{Gro-89}. Delzant and Py showed that if a K\"ahler group acts geometrically on a locally finite CAT(0) cube complex, then it is commensurable to a direct product of finitely many surface groups and a free abelian group \cite{DelPy-16}.

More generally, a close connection between K\"ahler groups acting on CAT(0) cube complexes and subgroups of direct products of surface groups has been observed starting with the work of Delzant and Gromov on cuts in K\"ahler groups \cite{DelGro-05} (see also \cite{Py-13, DelPy-16}). This led Delzant and Gromov to pose the question of which K\"ahler groups are subgroups of direct products of surface groups? Following the work of Bridson, Howie, Miller and Short \cite{BriHowMilSho-02, BriHowMilSho-09}, one knows that this question is intimately related to the question of finding K\"ahler groups which are not of \textit{finiteness type} $\mathcal{F}_r$ for some $r$, i.e. do not admit a classifying space with finite $r$-skeleton: any subgroup of a direct product of $k$ surface groups which is $\mathcal{F}_k$ is virtually a direct product of surface groups and finitely generated free groups.
 
 The first examples of K\"ahler subgroups of direct products of surface groups which are of type  $\mathcal{F}_{r-1}$ but not $\mathcal{F}_r$ ($r\geq 3$) were constructed by Dimca, Papadima and Suciu \cite{DimPapSuc-09-II}. Their class of examples has since been extended by Biswas, Mj and Pancholi \cite{BisMjPan-14} and by the author \cite{Llo-16-II}. All of these examples arise as kernels of surjective homomorphisms of the form $\pi_1(S_{g_1})\times \cdots \times \pi_1 (S_{g_r})\rightarrow \ZZ^2$ where $r\geq 3$ and $S_{g_i}$ is a closed Riemann surface of genus $g_i\geq 2$, $1\leq i \leq r$. Recently, examples of K\"ahler groups that are of type $\mathcal{F}_{r-1}$ but not of type $\mathcal{F}_r$, and which are not commensurable to any subgroup of a direct product of surface groups have been constructed by Bridson and the author \cite{BriLlo-16}. 
 
 We want to recall two key notions when studying subgroups of direct products of surface groups: A subgroup $H\leq G_1 \times \dots \times G_r$ of a direct product of $r$ groups $G_i$ is called \textit{full} if all intersections $G_i\cap H:= \left(1\times \dots \times 1 \times G_i \times 1 \times \dots \times 1\right) \cap H$ are non-trivial, and \textit{subdirect} if $p_i(H)= G_i$ for all $1\leq i \leq r$, with $p_i: G_1 \times \dots \times G_r\to G_i$ the projection. Their significance stems from the fact that every finitely presented subgroup of a direct product of surface groups with trivial center admits a universal embedding in a direct product of finitely many free groups and surface groups with full subdirect image \cite[Theorem C]{BriHowMilSho-13}.
 
 This paper consists of three parts. In the first part (Section 2) we develop a new construction method for K\"ahler groups. The groups obtained from this method arise as fundamental groups of fibres of holomorphic maps onto higher-dimensional complex tori. In the second and third part we address Delzant and Gromov's question. In the second part (Sections 3 -- 5) we apply our construction method to provide K\"ahler subgroups of direct products of surface groups that are not commensurable with any of the previous examples. These arise as kernels of a surjective homomorphism onto $\ZZ^{2k}$ and are \textit{irreducible}, i.e. do not decompose as direct product of two nontrivial groups (even virtually). The examples constructed in this work significantly extend the range of irreducible full subdirect K\"ahler subgroups of direct products of surface groups: all previous examples of such K\"ahler subgroups of a product of $r$ surface groups are either virtually a product of surface groups and a free abelian group, or of type $\mathcal{F}_{r-1}$, but not of type $\mathcal{F}_r$. Here we produce irreducible examples of type $\mathcal{F}_k$ and not of type $\mathcal{F}_{k+1}$ for all $2\leq k \leq r-1$, hence covering the full range of possible finiteness properties \cite{BriHowMilSho-02, BriHowMilSho-09}. In the third part (Sections 6 -- 9) we give new constraints on K\"ahler subgroups of direct products of surface groups. In particular, we show that if a full subdirect product of $r$ surface group is K\"ahler of type $\mathcal{F}_k$ with $k> \frac{r}{2}$ then it is virtually the kernel of an epimorphism from the product of surface groups onto a free abelian group of even rank.

One says that a surjective holomorphic map $h: X\rightarrow Y$ between compact complex manifolds has \textit{isolated singularities} if the critical locus of $h$ intersects each fibre (preimage of a point) in a discrete subset. The key result in our construction method is Theorem \ref{thmFiltVerGen}, a special case of which is:
 
 \begin{theorem}
 Let $X$ be a compact complex manifold of dimension $n+k$ and let $Y$ be a complex torus of dimension $k$. Let $h:X\rightarrow Y$ be a surjective holomorphic map with connected smooth generic fibre $H$. Assume that there is a filtration
 \[
  \left\{0\right\} \subset Y^0\subset Y^1 \subset \cdots \subset Y^{k-1}\subset Y^k=Y
 \]
of $Y$ by complex subtori $Y^l$ of dimension $l$ such that the projections
\[
h_l=\pi_l\circ h: X\rightarrow Y/Y^{k-l}=:Y_l
\]
have isolated singularities, where $\pi_l: Y\rightarrow Y_l$ is the holomorphic quotient homomorphism.

If $n=\mathrm{dim}_{\CC}H\geq 2$, then the map $h$ induces a short exact sequence 
\[
 1 \rightarrow \pi_1 (H) \rightarrow \pi_1 (X) \rightarrow \pi_1 (Y)= \ZZ^{2k}\rightarrow 1.
\]
Furthermore, we obtain that $\pi_i(X,H)=0$ for $2\leq i \leq \mathrm{dim}_{\CC}H$.
\label{thmFiltVer}
\end{theorem} 

Theorem \ref{thmFiltVer} and Theorem \ref{thmFiltVerGen} are generalisations of \cite[Theorem C]{DimPapSuc-09-II} and \cite[Theorem 2.2]{BriLlo-16}. We expect that our methods can be applied to construct interesting new classes of K\"ahler groups. Indeed we provide a first application in this work, by constructing new classes of K\"ahler subgroups of direct products of surface groups.

\begin{notation*}
Throughout this article $S=S_g$ will always denote a closed Riemann surface of genus $g\geq 2$ and $\Gamma=\Gamma_{g}=\pi_1 (S_{g})$ its fundamental group.
\end{notation*}

\begin{theorem}
Let $r\geq 3$, $r-2\geq k\geq 1$ and $g_1,\dots, g_r\geq 2$. Then there are an elliptic curve $E$, closed Riemann surfaces $S_i$ of genus $g_i$ and a surjective holomorphic map
\[
 h: S_1\times \dots \times S_r \to E^{\times k}
\]
satisfying the following properties:
\begin{enumerate}
 \item the smooth generic fibre $H$ of $h$ is connected and K\"ahler (in fact projective);
 \item the inclusion $H\hookrightarrow S_1 \times \dots \times S_r$ induces an embedding $\pi_1 (H)\leq \pi_1(S_1)\times \dots \times \pi_1(S_r)$ as an irreducible full subdirect product with $\pi_1 (H)=\ker(h_{\ast})$;
 \item $\pi_1(H)$ is of type $\mathcal{F}_{r-k}$ but not of type $\mathcal{F}_{r-k+1}$;
 \item there is no $(r-k+1)$-dimensional smooth complex subvariety $\iota: X\hookrightarrow R_1\times \dots \times R_r$ of a direct product of $r$ Riemann surfaces $R_i$ with $\iota_{\ast}(\pi_1(X))\cong\pi_1 (H)$.
\end{enumerate}
\vspace{.3cm}
  In particular, there is a K\"ahler subgroup of $\G_1\times \dots \times \G_r$ which is an irreducible full subdirect product of type $\mathcal{F}_m$ but not of type $\mathcal{F}_{m+1}$, for every $r-1\geq m\geq 2$.
  \label{thmIntroA}
\end{theorem}

Here we use the notation $E^{\times k}=\underbrace{E\times \dots \times E}_{\mbox{$k$ times}}$ for the Cartesian product of $k$ copies of $E$. The coabelian subgroups of direct products of surface groups form an important subclass of the class of all subgroups of direct products of surface groups. Indeed, in the case of three factors any finitely presented full subdirect subgroup of $D=\pi_1 (S_1)\times \pi_1 (S_2)\times \pi_1 (S_3)$ is virtually \textit{coabelian}, i.e. contains the derived subgroup $\left[D_0,D_0\right]$ of some $D_0\leq D$ of finite index; with more factors any full subdirect subgroup is virtually conilpotent \cite{BriHowMilSho-13}. We will give a more detailed discussion of subgroups of direct products of surface groups in Section \ref{secNotProd}. 
 
Theorem \ref{thmIntroA} shows that there are indeed K\"ahler groups covering the full range of possible finiteness properties of irreducible full subdirect products of surface groups. We will see that a modification of the construction used to prove Theorem \ref{thmIntroA} provides a second class of examples (see Theorem \ref{thmExtendedRange}). This class will show that the in particular part of Theorem \ref{thmIntroA} can also be proved by considering only holomorphic maps to a product of two elliptic curves. The reduction in dimension of the complex torus will mean that these different examples do not satisfy Theorem \ref{thmIntroA} (4). This shows that there is no direct correlation between the finiteness properties of a K\"ahler subgroup $G\leq \pi_1(S_1)\times \dots \times \pi_1(S_r)$ and the maximal dimension of a complex submanifold $X \subset S_1\times \dots \times S_r$ such that the inclusion map has image $G$ on fundamental groups.

Conversely, we address the question of finding constraints on K\"ahler subgroups of direct products of surface groups, or, more generally, on K\"ahler groups that admit homomorphisms to direct products of surface groups.  

\begin{definition}
 For a K\"ahler group $G$ we call a subgroup $H\leq G$ \textit{holomorphically coabelian} if there is a compact K\"ahler manifold $X$ with $G=\pi_1 (X)$ and a complex torus $T$ such that $H$ is the kernel of an epimorphism $\phi=h_{\ast}: \pi_1 (X) \to \pi_1 (T)=\ZZ^{2l}$ for $h: X\to T$ a holomorphic map.
 
 We say that $H\leq G$ is \textit{virtually holomorphically coabelian} if there are finite index subgroups $H_0\leq H$ and $G_0\leq G$ such that $H_0\leq G_0$ is holomorphically coabelian. 
\end{definition}

\begin{remark}
Note that every holomorphically coabelian subgroup $H\leq G=\pi_1 (X)$ is coabelian of even rank, that is, the kernel of an epimorphism from $G$ to a free abelian group of even rank. We will make use of this observation at several points in this work without further reference.
\end{remark}

We will show that under certain assumptions the image of a K\"ahler group under a homomorphism to a direct product of surface groups is virtually holomorphically coabelian.
\begin{theorem}
\label{thmNewCoab}
 Let $G=\pi_1 (X)$ with $X$ compact K\"ahler and let $\phi: G \rightarrow \overline{G}$ be a surjective homomorphism onto a subgroup $\overline{G}\leq \G_1 \times \dots \times \G_r$. Assume that $\phi$ has finitely generated kernel and that $\overline{G}$ is full and of type $\mathcal{F}_m$ for $m\geq 2$.
 
 Then, after reordering factors, there is $s\geq 0$ such that the projection $p_{i_1,\dots,i_k}(\overline{G})\leq \G_{g_{i_1}}\times \dots \times \G_{g_{i_k}}$ is virtually holomorphically coabelian, for all $k< 2m$ and all $1\leq i_1 < \dots < i_k\leq s$. Furthermore, the center $\mathrm{Z}(\overline{G})= \overline{G} \cap \left(\G_{g_{s+1}}\times \dots \times \G_{g_r}\right)\leq p_{s+1,\dots,r}(\overline{G})\cong \ZZ^{r-s}$ is a finite index subgroup. 
\end{theorem}

Combining Theorem \ref{thmNewCoab} with a study of the first Betti number of coabelian subdirect products of groups in Section \ref{secFinPropBetti}, allows us to show that there are non-K\"ahler subgroups of direct products of surface groups with interesting properties.

\begin{corollary}
\label{corNewCoabExs}
 Let $G=\ker (\psi)$ for $\psi: \G_{g_1}\times \dots\times \G_{g_r}\rightarrow \ZZ^{2l+1}$ an epimorphism. Then $\ker (\psi)$ is not K\"ahler.
\end{corollary}

\begin{corollary}
 For $r\geq 6$ and $g_1,\dots,g_r\geq 2$ there is a non-K\"ahler full subdirect product $G\leq \G_{g_1}\times \dots \times \G_{g_r}$ with even first Betti number.
 \label{corExEvenB1Intro}
\end{corollary}

\subsection*{Structure:}

This work is structured as follows: In Section \ref{secMainThm} we prove Theorem \ref{thmFiltVerGen}, which implies Theorem \ref{thmFiltVer}. In Sections \ref{secExamples} and \ref{secExCombined} we construct large new classes of K\"ahler subgroups of direct products of surface groups, which we use to prove Thereom \ref{thmIntroA}. In Section \ref{secNotProd} we show that these examples are irreducible and derive their precise finiteness properties. In Section \ref{secResCoabKGs} we study homomorphisms from K\"ahler groups to direct products of surface groups and prove Theorem \ref{thmNewCoab} and Corollary \ref{corNewCoabExs}. In Section \ref{secSESCoab} we study the first Betti number of coabelian subgroups of direct products of groups and prove Corollary \ref{corExEvenB1Intro}. We apply the results of Section \ref{secFinPropBetti} to obtain additional constraints on homomorphisms from K\"ahler groups to direct products of surface groups. In Section \ref{secConsGens} we consider the universal homomorphism from a K\"ahler group to a direct product of Riemann orbisurfaces and we explain why our constraints apply to it.

\begin{acknowledgements*}
I am very grateful to my advisor Martin Bridson for his generous support and the many very helpful discussions we had about the contents of this paper, and to Simon Donaldson for inspiring conversations about topics related to the contents of this paper. I would also like to thank the anonymous referees for their careful reading of the paper and the many valuable suggestions that improved the exposition, as well as some of the proofs and results.

This work was supported by a EPSRC Research Studentship, by the German National Academic Foundation, and by a public grant as part of the FMJH.

This article is a fundamentally rewritten and extended version of arXiv:1701.01163v2. Some of the new material contained in this work is based on results from the authors PhD thesis. 
\end{acknowledgements*}

\section{A new construction method}

\label{secMainThm}

Let $X$ and $Y$ be compact complex manifolds and let $f:X\rightarrow Y$ be a surjective holomorphic map. Recall that since $X$ and $Y$ are compact a sufficient condition for the map $f$ to have isolated singularities is that the set of singular points of $f$ intersects every fibre of $f$ in a finite set \cite[(2.7) \& (2.8)]{Loo-84}. Having isolated singularities yields strong restrictions on the topology of the fibres near the singularities. We will only make indirect use of these restrictions here, by applying Theorem \ref{thmC'}. For background on isolated singularities see \cite{Loo-84}.

Before we proceed we fix some notation: For a set $M$ and subsets $A,B\subset M$ we will denote by $A\setminus B$ the set theoretic difference of $A$ and $B$. If $M=T^n$ is an abelian group then we will denote by $A-B=\left\{a-b\mid a\in A, b\in B\right\}$ the group theoretic difference of $A$ and $B$ with respect to the additive group structure. We will be careful to distinguish $-$ from set theoretic $\setminus$.

In this section we generalise the following result of Dimca, Papadima and Suciu to maps onto higher-dimensional tori by proving Theorem \ref{thmFiltVer}.

\begin{theorem}[{\cite[Theorem C]{DimPapSuc-09-II}} ]
Let $X$ be a compact complex manifold and let $Y$ be a closed Riemann surface of genus at least one. Let $f: X\rightarrow Y$ be a surjective holomorphic map with isolated singularities and connected fibres. Let $\widehat{f}: \widehat{X}\rightarrow \widetilde{Y}$ be the pull-back of $f$ under the universal cover $p:\widetilde{Y} \rightarrow Y$ and let $H$ be the smooth generic fibre of $\widehat{f}$ (and therefore of $f$).

Then the following hold:
\begin{enumerate}
 \item $\pi_i(\widehat{X},H)=0$ for $ i \leq \mathrm{dim}_{\CC}H$;
 \item if $\mathrm{dim}_{\CC} X \geq 3$, then  $1\rightarrow \pi_1 (H)\rightarrow \pi_1 (X) \overset{f_{\ast}}\rightarrow \pi_1 (Y)\rightarrow 1$ is exact.
\end{enumerate}
\label{thmC'}
\end{theorem}

Note that in fact we will prove the more general Theorem \ref{thmFiltVerGen} from which Theorem \ref{thmFiltVer} follows immediately. These results can be seen as Lefschetz type results, since they say that in low dimensions the homotopy groups of the subvariety $H\subset \widehat{X}$ of complex dimension $n\geq 2$ coincide with the homotopy groups of $\widehat{X}$. The most classical Lefschetz type theorem is the Lefschetz Hyperplane Theorem (see \cite{GorMac-88} for a detailed introduction to Lefschetz type theorems).

\subsection{Fibrelong isolated singularities}
\label{secFibrelongSingBriLlo}

To prove Theorem \ref{thmFiltVerGen} we make use of a generalisation of Theorem \ref{thmC'} which relaxes the conditions on the singularities of $h$.
 
\begin{definition}
 Let $X$, $Y$ be compact complex manifolds. We say that a surjective map $h: X\rightarrow Y$ has \textit{fibrelong isolated singularities} if it factors as
 \[
  \xymatrix{ X \ar[r]^g \ar[rd]^h & Z\ar[d]^f \\ & Y\\}
 \]
 where $Z$ is a compact complex manifold, $g$ is a holomorphic submersion with connected fibres (and thus defines a locally trivial fibration), and $f$ is holomorphic with isolated singularities. We denote by $F$ a fibre of $g$.
 \label{defAIS}
\end{definition}

For holomorphic maps with connected fibrelong isolated singularities, Bridson and the author proved

\begin{theorem}[{\cite[Theorem 2.2]{BriLlo-16}}]
\label{thm2}
Let $Y$ be a closed Riemann surface of positive genus and let $X$ be a compact K\"ahler manifold. Let $h:X\rightarrow Y$ be a surjective holomorphic map with connected generic (smooth) fibre $\overline{H}$.

If $h$ has fibrelong isolated singularities, $g$ and $f$ are as in Definition \ref{defAIS}, and $f$ has connected fibres of complex dimension $m\geq 2$, then the sequence 
\[  
1 \rightarrow \pi_1 (\overline{H})\rightarrow \pi_1 (X) \overset{h_{\ast}}\rightarrow \pi_1 (Y)\rightarrow 1
\]
is exact.
\end{theorem}

We will also need the following proposition. 

\begin{proposition}[{\cite[Proposition 2.3]{BriLlo-16}}]
\label{prop1part2}
Under the assumptions of Theorem \ref{thm2}, if the map $\pi_2 (Z)\to \pi_1 (F)$ associated to the fibration $g:X\to Z$ is trivial, then the long exact sequence induced by the fibration $F\hookrightarrow \overline{H}\rightarrow H$ reduces to a short exact sequence
\[
1\rightarrow \pi_1 (F)\rightarrow \pi_1 (\overline{H})\rightarrow \pi_1 (H)\rightarrow 1.
\]
If, in addition, the fibre $F$ is aspherical, then $\pi_ i (\overline{H}) \cong \pi_i (H) \cong \pi_i X$ for $2\leq i \leq m - 1$.
\end{proposition}

Note that the hypothesis on $\pi_2 (Z)\to \pi_1(F)$  is automatically satisfied if $\pi_1(F)$ does not contain a non-trivial normal abelian subgroup. This is the case, for example, if $F$ is a direct product of hyperbolic surfaces.

\subsection{Restrictions on $h:X\rightarrow Y$ for higher-dimensional tori}
\label{sec:Restrictions}

Let $X$ be a compact complex manifold and let $Y$ be a complex torus of dimension $k$. Let $h: X\rightarrow Y$ be a surjective holomorphic map. Assume that there is a filtration
\[
 \left\{0\right\} \subset Y^0\subset Y^1 \subset \cdots \subset Y^{k-1}\subset Y^k=Y
\]
of $Y$ by complex subtori $Y^l$ of dimension $l$, $0\leq l \leq k$. We denote $Y_l:= Y/Y^{k-l}$ the quotient torus for $0\leq l \leq k$. Let $\pi_l: Y\rightarrow Y_l$ be the canonical holomorphic projection. 

\begin{remark}
 Note that if $Y$ is an abelian variety the existence of such a filtration implies that $Y$ is isogeneous to a direct product of elliptic curves by Poincar\'e's Reducibility Theorem (for instance \cite[Theorem 5.5]{BirLan-99}). This is, however, false in general for non-abelian varieties, even for extensions of elliptic curves by elliptic curves \cite[Section 3]{OorZar-96}, \cite[Section 1.6]{BirLan-99}.
\end{remark}

Assume that the maps $h_l=\pi_l\circ h: X\rightarrow Y_l$ have connected fibres and fibrelong isolated singularities. By definition, there are compact complex manifolds $Z_l$ such that the $h_l$ factor as
\[
 \xymatrix{ X \ar[r]^{g_l}\ar[rd]_{h_l} & Z_l\ar[d]^{f_l} \\ & Y_l,}
\]
with $g_l$ a holomorphic submersion with connected fibres, and $f_l$ surjective holomorphic with isolated singularities and connected fibres. Assume further that the smooth fibre $F_l$ of $g_l$ is aspherical. We denote by $\overline{H}_l$ the connected smooth generic fibre of $h_l$ and by $H_l$ the connected smooth generic fibre of $f_l$. 

Let $x \in Y$ be a generic point and let $x_l=\pi_l(x)$ be its image in $Y_l$ for $1\leq l \leq k$. Note that $\pi_k: Y \to Y_k$ is the identity map, thus identifying $Y$ with $Y_k$ and $x$ with $x_k$. Since the $\pi_l: Y\to Y_l$ are surjective, the $x_l\in Y_l$ are also generic points. Since the singular values of $h_l$ are contained in a proper subvariety of $Y_l$ for $1\leq l \leq k$, it follows that $x_l\in Y_l$ is a regular value for $1\leq l \leq k$.
 
 By definition, the smooth generic fibres $\overline{H}_l=h_l^{-1}(x_l)$ of $h_l$ form a nested sequence
 \[
  \overline{H}= \overline{H}_k\subset \overline{H}_{k-1}\subset \cdots \subset \overline{H}_0=X.
 \]

Denote by $E_l= x_l + Y^{k+1-l}/Y^{k-l}\subset Y_l$ the translate of the quotient group $Y^{k+1-l}/Y^{k-l}$ passing through $x_l$.
 
\begin{lemma}
 Assume that for every $l$ the map $h_l$ has connected fibres and fibrelong isolated singularities. Then the restriction maps $h_l|_{\overline{H}_{l-1}}: \overline{H}_{l-1}\to E_l$ are also surjective holomorphic with fibrelong isolated singularities and connected smooth generic fibre $\overline{H}_{l-1}$.
\label{lem:RestFLIS}
\end{lemma}
\begin{proof}
Denote by $\pi_{l-1,l}: Y_l \to Y_{l-1}$ the canonical projection and observe that $h_{l-1}=\pi_{l-1,l}\circ h_l$. By definition $\overline{H}_{l-1}=h_{l-1}^{-1}(x_{l-1})= h_l^{-1}(\pi_{l-1,l}^{-1}(x_{l-1}))=h_l^{-1}(E_l)$, implying that $h_l|_{\overline{H}_{l-1}}$ is surjective holomorphic with connected smooth generic fibre $\overline{H}_l$.

The factorization $h_l=f_l\circ g_l$ induces a factorization
\begin{equation}
 \xymatrix{ \overline{H}_{l-1} \ar[r]^{g_l|_{\overline{H}_{l-1}}}\ar[rd]_{h_l|_{\overline{H}_{l-1}}} & g_l(\overline{H}_{l-1})\ar[d]^{f_l|_{g_l(\overline{H}_{l-1})}} \\ & E_l.}
\label{eqnRestIsolSing}
\end{equation} 
Note that $g_l(\overline{H}_{l-1})=f_l^{-1}(E_l)=f_l^{-1}\left(\pi_{l-1,l}^{-1}(x_{l-1})\right)$. By genericity, $x_{l-1}$ is a regular value of $\pi_{l-1,l}\circ f_l : Z_l \to Y_{l-1}$, implying that $g_l(\overline{H}_{l-1})$ is a smooth complex manifold. Using that $h_l=f_l\circ g_l$ has fibrelong isolated singularities, it is now easy to check that the same holds for $h_l|_{\overline{H}_{l-1}}$ with respect to the factorization \eqref{eqnRestIsolSing}.
\end{proof}

\subsection{The Main Theorem}
We can now state and prove the Main Theorem of this section. 
\begin{theorem}
 Assume that $h:X\rightarrow Y$ has all the properties described in Paragraph \ref{sec:Restrictions}, that the induced map $\pi_2 (\overline{H}_{l-1})\rightarrow \pi_1 (F_l)$ is trivial for $1\leq l\leq k$, and that $n:= \mathrm{min}_{1\leq l \leq k} \mathrm{dim}_{\CC} H_l\geq 2$. 
 
 Then the map $h$ induces short exact sequences
 \[
  1\longrightarrow \pi_1 (\overline{H})\longrightarrow \pi_1 (\overline{H}_{l}) \stackrel{h|_{\overline{H}_l,\ast}}{\longrightarrow} \pi_1 (x_k+Y^{k-l}) \cong \ZZ^{2(k-l)} \longrightarrow 1
 \] 
and $\pi_i(\overline{H}_l)\cong \pi_i (X)$ for $2\leq i \leq n-1$ and $0\leq l \leq k$.
\label{thmFiltVerGen}
\end{theorem}
Note that Theorem \ref{thmFiltVer} is the special case of Theorem \ref{thmFiltVerGen} with $Z_l=X$ and $g_l=id_X$ for $1\leq l \leq k$. 

We emphasize the following important observation, which we will apply in Section \ref{secExCombined}.
\begin{addendum}
\label{addFiltVerGenWeak}
In fact we will see that the proof and conclusions of Theorem \ref{thmFiltVerGen} remain correct when we replace the assumptions in Paragraph \ref{sec:Restrictions} by the weaker assumption that there is a point $x_k\in Y$ such that
\begin{itemize}
\item $x_l=\pi_l(x)$ is a regular value of $h_l$ for $1\leq l \leq k$;
\item the corestrictions of the maps $h_l$ to $E_l=x_l+ Y^{k-l+1}/Y^{k-l}$ have fibrelong isolated singularities; 
\item and, in the corresponding factorization $h_l|_{\overline{H}_{l-1}}=f_l\circ g_l$, the maps $f_l$ have connected smooth generic fibres $H_l$ and the maps $g_l$ have aspherical fibres $F_l$.
\end{itemize}
 The generic choice of $x$ in Paragraph \ref{sec:Restrictions} and Lemma \ref{lem:RestFLIS} show that all of these conditions are satisfied under the assumptions of Theorem \ref{thmFiltVerGen}.
\end{addendum}

\begin{proof}[Proof of Theorem \ref{thmFiltVerGen}]
 The proof is by induction on $l$ starting with $l=k$ and decreasing to $l=0$, The induction step will be based on applying Theorem \ref{thm2} and Proposition \ref{prop1part2}. 
 
Since $H_l$ is connected and $\mathrm{dim}_{\CC} H_l\geq n\geq 2$, Theorem \ref{thm2}, Proposition \ref{prop1part2} and Lemma \ref{lem:RestFLIS} imply that the restrictions $h_l|_{\overline{H}_{l-1}}$ induce short exact sequences
\begin{equation}
1\rightarrow \pi_1 (\overline{H}_l) \rightarrow \pi_1 (\overline{H}_{l-1})\stackrel{h_{l\ast}}{\rightarrow} \pi_1 \left(E_l\right) = \ZZ^2\rightarrow 1
\label{eqnSES1}
\end{equation}
and that $ \pi_i(\overline{H}_{l-1})\cong  \pi_ i(\overline{H}_l)$ for $2\leq i \leq \mathrm{dim} H_l-1$, where $1\leq l \leq k$. In particular, we obtain that $\pi_ i(\overline{H}_{l-1}) \cong \pi_i(\overline{H}_l)$ for $2\leq i \leq n-1$.
 
 Hence, we are left with proving that the short exact sequences in \eqref{eqnSES1} induce short exact sequences
\begin{equation}
1\longrightarrow \pi_1 (\overline{H})\longrightarrow \pi_1 (\overline{H}_{l}) \stackrel{h|_{\overline{H}_l,\ast}}{\longrightarrow} \pi_1 (x_k+Y^{k-l}) \cong \ZZ^{2(k-l)} \longrightarrow 1,
\label{eqnSESInd}
\end{equation}
for $0\leq l \leq k$. The case $l=k$ is trivial. 

Now assume that the sequence \eqref{eqnSESInd} is exact for $l=l_0$ with $0< l_0 \leq k$ and consider the commutative diagram of topological spaces
\[
\xymatrix{
\overline{H} \ar@{^{(}->}[r] & \overline{H}_{l_0-1}=h^{-1}\left(x_k+ Y^{k-l_0+1}\right) \ar@{->>}[r] ^-{h} & x_k+ Y^{k-l_0+1}\\
\overline{H} \ar@{^{(}->}[r] \ar[u]^{=} & \overline{H}_{l_0}=h^{-1}\left(x_k+ Y^{k-l_0}\right) \ar@{->>}[r]^-{h} \ar@{^{(}->}[u] & x_k+ Y^{k-l_0}. \ar@{^{(}->}[u] \\
} 
\]
It induces a commutative diagram of fundamental groups 
\begin{equation}
\xymatrix{
1\ar[r] & \pi_1 (\overline{H}) \ar@{^{(}->}[r] & \pi_1 (\overline{H}_{l_0-1}) \ar@{->>}[r] ^-{h_{\ast}} & \pi_1(x_k+Y^{k-l_0+1})= \ZZ^{2(k-l_0+1)}\ar[r] &1\\
1\ar[r] & \pi_1 (\overline{H}) \ar@{^{(}->}[r] \ar[u]^{=} & \pi_1 (\overline{H}_{l_0}) \ar@{->>}[r]^-{h_{\ast}} \ar@{^{(}->}[u] & \pi_1(x_k+Y^{k-l_0})=\ZZ ^{2(k-l_0)} \ar@{^{(}->}[u]\ar[r] &1, \\
}
\label{eqndiaggps}
\end{equation}
where injectivity of the vertical map in the middle follows from \eqref{eqnSES1}. Note that the bottom row of \eqref{eqndiaggps} is exact by induction hypothesis. We will now deduce that the top row is also exact.

Exactness at $\pi_1(\overline{H})$ follows from injectivity of $\pi_1 (\overline{H}_{l_0}) \hookrightarrow \pi_1 (\overline{H}_{l_0-1})$.

For exactness at $\pi_1 (x_k+Y^{k-l_0+1})$ observe that, by the Ehresmann Fibration Theorem, the fibration $\overline{H}_{l_0-1}\rightarrow x_k+Y^{k-l_0+1}$ restricts to a locally trivial fibration $\overline{H}_{l_0-1}^*\rightarrow (x_k+Y^{k-l_0+1})^*$  with connected fibre $\overline{H}$ over the complement $(x_k+Y^{k-l_0+1})^*$ of the subvariety of critical values of $h$ in $x_k+Y^{k-l_0+1}$. Hence, the induced map $\pi_1 (\overline{H}_{l_0-1}^*)\rightarrow \pi_1 (x_k+Y^{k-l_0+1})^*$ on fundamental groups is surjective. Since the complements $\overline{H}_{l_0-1}\setminus \overline{H}_{l_0-1}^*$ and $(x_k+Y^{k-l_0+1})\setminus (x_k+Y^{k-l_0+1})^*$ are contained in complex analytic subvarieties of complex codimension at least one, the induced map $\pi_1 (\overline{H}_{l_0-1})\rightarrow \pi_1 (x_k+Y^{k-l_0+1})$ is surjective.

The induction step is thus completed by the following result:
\begin{lemma}
 The top horizontal sequence in Diagram \eqref{eqndiaggps} is exact at $\pi_1(\overline{H}_{l_0-1})$.
\end{lemma}
\begin{proof}
It is clear that $\pi_1 (\overline{H})\leq \mathrm{ker}\left(\pi_1 (\overline{H}_{l_0-1})\rightarrow \pi_1(x_k+Y^{k-l_0+1})\right)$. Hence, the only thing that we need to prove is that $\pi_1 (\overline{H})$ contains 
$\mathrm{ker}\left(\pi_1 (\overline{H}_{l_0-1})\rightarrow \pi_1(x_k+Y^{k-l_0+1})\right).$

Let $g\in \mathrm{ker}\left(\pi_1 (\overline{H}_{l_0-1})\stackrel{h_{\ast}}{\rightarrow} \pi_1(x_k+Y^{k-l_0+1}) \right)$. Then
\[
g\in \mathrm{ker}\left( \pi_1 (\overline{H}_{l_0-1}) \stackrel{h_{l_0\ast}}{\rightarrow} \pi_1 \left(E_{l_0}\right)\right),
\]
since the map $h_{l_0\ast}$ factors through $h_{\ast}: \pi_1 (\overline{H}_{l_0-1})\rightarrow \pi_1(x_k+Y^{k-l_0+1})$.

By exactness of \eqref{eqnSES1} for $l=l_0$, this implies that there is $\tilde{g}\in \pi_1(\overline{H}_{l_0})$ with $\iota_{l_0\ast}(\tilde{g})=g$, where $\iota_{l_0}: \overline{H}_{l_0}\hookrightarrow \overline{H}_{l_0-1}$ is the inclusion map. It follows from commutativity of the diagram of groups \eqref{eqndiaggps} and injectivity of the vertical maps that $\tilde{g}\in \mathrm{ker}(\pi_1 (\overline{H}_{l_0})\rightarrow \pi_1 (x_k+Y^{k-l_0}))$.

The induction assumption now implies that $\tilde{g}\in \mathrm{Im} (\pi_1 (\overline{H}) \rightarrow \pi_1 (\overline{H}_{l_0}))$. Hence, $g\in \pi_1 (\overline{H})$, completing the proof of the lemma.
\end{proof}
It follows that the Sequence \eqref{eqnSESInd} is a short exact sequence for $0\leq l \leq k$, thus completing the proof.
\end{proof}

\section{A class of higher dimensional examples}
\label{secExamples}

In this section we will construct a general class of examples of K\"ahler subgroups of direct products of surface groups arising as kernels of homomorphisms onto $\ZZ^{2k}$ for any $k\geq 1$. 

Let $E=\CC/\Lambda$ be an elliptic curve, let $r\geq 3$ and let
\[
 \alpha _i: S_i \rightarrow E
\]
be branched holomorphic coverings by closed Riemann surfaces $S_i$ of genus $\geq 2$ for $1\leq i \leq r$. 

Our groups will be the fundamental groups of the fibres of surjective holomorphic maps from the direct product $S_1\times \cdots \times S_r$ onto the $k$-fold direct product $E^{\times k}$ of $E$ with itself. For vectors $w_1,\cdots w_n\in \ZZ^k$ we will use the notation $\left(w_1\mid \cdots \mid w_n\right)$ to denote the $k\times n$-matrix with columns $w_i$ and $w_{i,j}$ to denote the $j$-th entry of $w_i$. 

For $n\geq 1$ we consider the $\ZZ$-module $E^{\times n}=\CC^n/\Lambda^n$. The $\ZZ$-module homomorphisms $\mathrm{Hom}(E^{\times n_1},E^{\times n_2})$ are precisely the $\CC$-linear maps $B: \CC^{n_1} \to \CC^{n_2}$ with $B(\Lambda^{n_1})\leq \Lambda^{n_2}$. In particular, we have $\mathrm{M}_{n_2,n_1}(\ZZ) \leq \mathrm{Hom}\left(E^{\times n_1},E^{\times n_2}\right)$ for $M_{n_2,n_1}(\ZZ)=\ZZ^{n_2\times n_1}$ the set of $n_2\times n_1$-matrices with coefficients in $\ZZ$, and $Gl(n,\ZZ)\leq \mathrm{End}(E^{\times n})$ if $n=n_1=n_2$.

We say that a set of vectors $\mathcal{C}=\left\{v_1,\cdots, v_r\right\}\subset\ZZ^k$ has property 
\begin{enumerate}
\item[(P)] if there is $1\leq i_1 < \dots <i_k \leq r$ such that $\left\{v_{i_1},\dots,v_{i_k}\right\}$ is a $\ZZ$-basis for $\ZZ^k$ and any choice of $k$ vectors in $\mathcal{C}$ is linearly independent; and
\item[(P')] if $\mathcal{C}$ has property (P), and in addition $i_j=j$ and $\left\{v_1,\dots v_k\right\}$ is the standard basis for $\ZZ^k$.
\end{enumerate}
For any set $\mathcal{C}=\left\{v_1,\cdots,v_r\right\}\subset \ZZ^k$ and $B=\left(v_1\mid \cdots \mid v_r\right)$ we can define the holomorphic map 
\[
h=B\circ \left(\alpha_1,\cdots,\alpha_r\right)=\sum \limits_{i=1}^r v_i \cdot \alpha_i :S_1\times \cdots \times S_r\rightarrow E^{\times k}.
\]

We will be interested in maps $h$ for which the set $\mathcal{C}$ has property (P). Note that, after reordering factors and adjusting by a biholomorphic automorphism of $E^{\times k}$, say $A\in \rm{GL}(k,\ZZ)$, we may in fact assume that $\mathcal{C}$ has property (P') if it has property (P). The following observation shows that such maps exist.

\begin{remark}
 Note that for $k\geq 1$, the $\ZZ$-module $\ZZ^k$ is not a finite union of $\ZZ$-submodules of rank $\leq (k-1)$. Thus, for every $r\geq k$ we can inductively construct a subset $\mathcal{C}=\left\{v_1,\cdots,v_r\right\}\subset \ZZ^k$ which satisfies property (P'). In particular, we can always complete a set $\left\{v_1,\dots,v_m\right\}\subset \ZZ^k$ with property (P') for $m\leq r$ to a set $\left\{v_1,\dots,v_r\right\}$ with property (P').
 \label{rmkLinIndep}
\end{remark}

The main result of this section is:

\begin{theorem}
 Let $1\leq k \leq r-2$. Assume that $\mathcal{C}$ satisfies property (P'), and that $\alpha_1,\cdots,\alpha_k$ are surjective on fundamental groups. 
 
 Then $h$ satisfies the hypotheses of Theorem \ref{thmFiltVerGen}, and $\overline{H}$ has the following properties:
 \begin{enumerate}
\item $\overline{H}$ is an $(r-k)$-dimensional K\"ahler (projective) submanifold of a direct product of $r$ Riemann surfaces;
  \item the inclusion $\overline{H}\hookrightarrow S_1\times \dots \times S_r$ induces an embedding $\pi_1 (\overline{H})\leq \pi_1 (S_1)\times  \cdots \times \pi_1 (S_r)$ as an irreducible full subgroup with $\pi_1(\overline{H})=\ker (h_{\ast})$;
  \item $\pi_1 (\overline{H})$ is of finiteness type $\mathcal{F}_{r-k}$, but not of type $\mathcal{F}_{r-k+1}$;
  \item there is no $(r-k+1)$-dimensional smooth projective subvariety $X\stackrel{\iota}{\hookrightarrow} R_1 \times \dots \times R_r$ of a direct product of $r$ Riemann surfaces $R_i$ with $\iota_{\ast}(\pi_1 (X)) = \pi_1 (\overline{H})$. 
 \end{enumerate}
Moreover, $\pi_j (\overline{H}) = 0$ for $2\leq j \leq r-k-1$.
 \label{thmExsGenClass}
\end{theorem}

The rest of this section will be dedicated to the proof of this result. In particular, we will assume from now that $\mathcal{C}$ and $h$ satisfy the assumptions of Theorem \ref{thmExsGenClass}. We consider the filtration $Y^l=E^{\times l}\times \left\{0\right\}$ of $E^{\times k}$ where $\pi_l: E^{\times k}\rightarrow Y_l=Y^k/Y^{k-l}= \left\{0\right\}  \times E^{\times l}$ is the projection onto the last $l$ coordinates and denote $h_l=\pi_l \circ h$, for $0\leq l \leq k$. 

Due to the assumptions on $\mathcal{C}$ the map $h_l$ factors as $h_l= f_l\circ g_l$ for $1\leq l \leq k$ with
 \[
  f_l=\pi_l \circ \left(v_{k-l+1} \mid \cdots  \mid v_r\right)\circ(\a _{k-l+1},\cdots, \a _r):  S_{k-l+1}\times \cdots \times S_r\rightarrow Y_l= \left\{0\right\} \times E^{\times l}  
 \]
 and
 \[
 g_l: S_1\times \cdots \times S_r\rightarrow  S_{k-l+1} \times \cdots \times S_r
 \]
 the canonical projection with fibre $F_l:= S_{1}\times \cdots \times S_{k-l}$ a product of closed hyperbolic surfaces (It follows from the fact that $v_1,\cdots, v_k\in \ZZ^k$ is a standard basis of $\ZZ^k$ that $h_l=f_l\circ g_l$ for $1\leq l \leq k$).

We will first check that $h$ satisfies all hypotheses of Theorem \ref{thmFiltVerGen}. The following is a natural and straight-forward generalization of \cite[Lemma 2.1]{Llo-16-II}.
\begin{lemma}
\label{lemNewConnFib}
 Let $X$ be a connected compact complex manifold, $E$ an elliptic curve and $f: X \rightarrow E$ a surjective holomorphic map. If there is a subgroup $A=\ZZ^2 \leq \pi_1 (X)$ such that $f_{\ast} (A) \leq \pi_1 (E)$ is a finite index subgroup and $f_{\ast} (\pi_1 (X)) = \pi_1 (E)$, then $f$ has connected and non multiple fibres.
\end{lemma}

\begin{proof}
 Since $f$ is proper, by Stein factorisation and \cite[p. 104]{Cat-08}, there is a closed Riemann orbisurface $S_{g,\mm}$, such that $f$ factors as
 \[
 \xymatrix{ X \ar[r]^{h_1} \ar[rd]_f & S_{g,\mm} \ar[d] ^{h_2}\\ & E,}
 \]
 where $h_1$ is holomorphic with connected and non multiple fibres and $h_2$ is holomorphic and finite-to-one; for the definition of a closed Riemann orbisurface refer to Section \ref{secResCoabKGs}. In particular, $h_2$ induces a (possibly ramified) covering map between the underlying Riemann surface $S_g$ of $S_{g,\mm}$ and $E$. 
 
 By assumption, the restriction $f_{\ast}|_A$ is injective. Thus, $h_{1,\ast} (A)\leq \pi_1^{orb} (S_{g,\mm})$ defines a $\ZZ^2$-subgroup of $\pi_1^{orb} (S_{g,\mm})$. It follows that $S_g=S_{g,\mm}$ is an elliptic curve and $h_2: S_{g,\mm}\to E$ is an unramified cover. Surjectivity of $f_{\ast}: \pi_1^{orb} (S_{g,\mm})\to \pi_1 (E)$ implies that $S_{g,\mm}=E$. Hence, $f$ has connected and non multiple fibres.
\end{proof}

\begin{lemma}
\label{propConnGenExs}
  Under the assumptions of Theorem \ref{thmExsGenClass}, the maps $h$, $h_l$, $f_l$ and $g_l$, $1\leq l \leq k$, have connected fibres.
\end{lemma}
\begin{proof}
It suffices to consider the maps $f_l$, as the maps $g_l$ clearly have connected fibres and the connectedness of the fibres of the maps $h_l$ then follows from the identity $h_l=f_l\circ g_l$.

Choose a generic point $x=x_k\in E^{\times k}=Y^k$ as in Section \ref{sec:Restrictions}, let $H_l=f_l^{-1}(x_l)$, $\overline{H}_l=h_l^{-1}(x_l)$ be the corresponding fibres, and denote $E_l:= x_l+Y^{k-l+1}/Y^{k-l}$. The proof is by induction on $l$.

First consider $l=1$. Then we have
\[
f_1= \pi_1 \circ \left(v_k\mid \dots \mid v_r\right) \circ \left(\alpha_k,\dots, \alpha_r\right): S_k\times \dots \times S_r\to Y_1=\left\{0\right\} \times E.
\]
In particular, we have $f_1 =\sum_{j=k}^r v_{k,j}\cdot \alpha_j$. By assumption on $\mathcal{C}$, $r-k+1\geq 2$, and we have $v_{k,k}=1$ and $v_{k,j}\neq 0$ for $j\geq k$. Since $\alpha_k: S_k\to E$ is surjective on fundamental groups, the same holds for $f_1$. Furthermore, the restriction $f_1|_{S_j}$ defines a branched covering of $E$ for $k\leq j \leq r$. Hence, $f_{1,\ast}(\pi_1 (S_j))\leq \pi_1 (\left\{0\right\} \times E)$ is a finite index subgroup for $k\leq j \leq r$. It follows that $f_1$ satisfies all conditions of Lemma \ref{lemNewConnFib} and therefore has connected fibres.

Let now $2\leq l \leq k$. By choice of $x_k$, the corestriction
\[
f_l|_{f_l^{-1}(E_l)}:f_l^{-1}(E_l) \to E_l=x_l+ Y^{k-l+1}/Y^{k-l}
\]
has smooth generic fibre $H_l$. Since by definition $\pi_l(v_{k-l+1})=\left(\begin{array}{c}1\\0\\ \vdots\\ 0\end{array}\right)\in\left\{0\right\}\times \ZZ^l$, we obtain
\[
f_l^{-1}(E_l)=S_{k-l+1}\times f_{l-1}^{-1}(x_{l-1})= S_{k-l+1}\times H_{l-1},
\]
with $H_{l-1}$ smooth and connected by induction assumption.

Choose a basepoint $z_0\in S_{k-l+1}$. By assumption on $\mathcal{C}$, the set $\pi_l(\left\{v_{k-l+2},\dots,v_r\right\})$ spans $\mathds{C}^l$. Thus, the map
\[
\sum_{j=k-l+2}^r \pi_l(v_j)\cdot \alpha_j: S_{k-l+2}\times \dots \times S_{r} \to Y_l = \left\{0\right\}\times  E^{\times l}
\]
is a surjective holomorphic map. Thus, the same holds for the restriction $$f_l|_{\left\{z_0\right\} \times S_{k-l+2}\times \dots \times S_{r}}.$$ It follows that 
\[
f_l|_{\left\{z_0\right\} \times H_{l-1}} : \left\{z_0\right\} \times H_{l-1}\to E_l
\]
is a surjective holomorphic map between compact complex manifolds and therefore $$f_{l,\ast}(\pi_1 (H_{l-1}))\leq \pi_1 (E)$$ is a finite index subgroup. On the other hand $\alpha_{k-l+1}$ is surjective on fundamental groups by assumption. Hence, $f_l|_{f_l^{-1}(E_l)}$ satisfies the assumptions of Lemma \ref{lemNewConnFib}. Therefore $f_l$ has connected smooth generic fibres, completing the induction.
\end{proof}

We can give a precise description of the set of singular values of $f_l$.

\begin{lemma}
A point $(z_{k-l+1},\cdots, z_r)\in  S_{k-l+1}\times \cdots \times S_{r}$ is a critical point of $f_l$ if and only if $z_i$ is a critical point of $\alpha_i$ for at least $r-k+1$ of the $z_i$, where $k-l+1\leq i\leq r$. 

In particular, the set of critical values $C_{f_l}$ of $f_l$ is the union $C_{f_l}=\bigcup \limits_{i\in I_l} B_{l,i}$ of a finite number of $(l-1)$-dimensional submanifolds $B_{l,i}\subset S_{k-l+1}\times \cdots \times S_{r}$ with the property that for every $k-l+1\leq j\leq r$ the projection of $B_{l,i}$ onto $S_{j}$ is either surjective or has finite image. 
\label{lemCritValfl}
\end{lemma}
\begin{proof}
Consider the differential
 \begin{equation}
  \D f_l = \left(\D \pi_l (v_{k-l+1})\cdot \mathrm{d} \a_{k-l+1},   \dots  ,\D\pi_l(v_r) \cdot \mathrm{d} \a_r\right).
 \label{eqnCritPt} 
 \end{equation}

Note that by definition of $\pi_l$ the vector $\D \pi_l(v_i)$ is the vector in $\ZZ^{l}$ consisting of the last $l$ entries of $v_i$. By property (P'), any $k$ vectors in $\mathcal{C}$ form a linearly independent set. Furthermore, we chose $\mathcal{C}$ such that $\mathcal{E}_1= \left\{v_1,\dots, v_k\right\}$ is the standard basis of $\ZZ^k$. This implies that the set
\[
\mathcal{C}_l = \left\{ \D \pi_l (v_{k-l+1}), \dots, \D\pi_l(v_r)\right\} \subset \ZZ^l
\]
also has property (P'). In particular, any choice of $l$ vectors in $\mathcal{C}_l$ forms a linearly independent set.

It thus follows from \eqref{eqnCritPt} that the critical points of $f_l$ are precisely the points $(z_{k-l+1},\cdots, z_r)\in  S_{k-l+1}\times \cdots \times S_{r}$ such that $z_i$ is a critical point of $\alpha_i$ for at least $r-k+1$ of the $z_i$, proving the first part of the statement. The second part is an immediate consequence of the first part.
\end{proof}

\begin{proposition}
 For all $l$ the map $h_l$ has fibrelong isolated singularities. More precisely, the factorisation $h_l=f_l\circ g_l$ satisfies that $g_l$ is a locally trivial fibration and $f_l$ has isolated singularities. Furthermore, the dimension of a smooth generic fibre $H_l$ of $f_l$ is $r-k$ for $1\leq l \leq k$.
 \label{propIsolSingGen}
\end{proposition}

\begin{proof}
Recall that by definition of $f_l$ and $g_l$ we have $h_l=f_l\circ g_l$, and the map $g_l$ has fibres $F_l= S_1\times \dots \times S_{k-l}$ a direct product of Riemann surfaces. In particular, it is a locally trivial fibration with connected fibres. 

Lemma \ref{lemCritValfl} provides us with an explicit description of the set of critical points of $f_l$.  Linear independence of any $l$ vectors in the set $\mathcal{C}_l$, defined in the proof of Lemma \ref{lemCritValfl}, implies that the restriction of $f_l$ to any of the $B_{l,i}$ is locally (and thus globally) finite-to-one. Hence, the intersection $C_{f_l}\cap H_{l,y}$ is finite for any fibre $H_{l,y}=f_l^{-1}(y)$, $y\in \left\{0\right\}\times E^{\times l}$. In particular, the map $f_l$ has isolated singularities. 

The assertion that a smooth generic fibre $H_l$ of $f_l$ has dimension $r-(k-l)- l=r-k$ follows immediately from the definition of $f_l$, completing the proof.
\end{proof}

\begin{proof}[Proof of Theorem \ref{thmExsGenClass}]
It follows from Lemma \ref{propConnGenExs} and Proposition \ref{propIsolSingGen} that $h$ satisfies the hypotheses of Theorem \ref{thmFiltVerGen} for $n=r-k$. Thus, (2) with the exception of irreducibility and the moreover part follow (after observing that the image of $\pi_1 (\overline{H})$ is full, because the restriction of $h_{\ast}$ to each factor has non-trivial kernel). 
Finally, (3), (4) and irreducibility follow from Lemma \ref{lemNotProdNotFr} in Section \ref{secNotProd}.
\end{proof}

\begin{proof}[Proof of Theorem \ref{thmIntroA}]
The only part of Theorem \ref{thmIntroA} that is not a direct consequence of Theorem \ref{thmExsGenClass} and its proof is that for a suitable choice of the map $h$ the group $\pi_1 (\overline{H})\leq \pi_1 (S_1)\times \dots \times \pi_1(S_r)$ is a subdirect product. Considering the set $\mathcal{C}$ such that $\left\{v_1,\dots,v_k\right\}$ is a standard basis of $\ZZ^k$ and $v_{k+1}=v_1+\dots +v_k$ we see that the complement $\mathcal{C}\setminus \left\{v_i\right\}$ of any vector contains a basis of $\ZZ^k$. Choosing $\alpha_{k+1}$ such that it induces an epimorphism on fundamental groups then assures that the restriction of $h_{\ast}$ to $\pi_1 (S_{1})\times \dots \times \pi_1 (S_{i-1})\times 1 \times \pi_1 (S_{i+1})\times \dots \times \pi_1 (S_{r})$ is surjective on fundamental groups for $1\leq i \leq r$ (see proof of Lemma \ref{lemNotProdNotFr} for more details on $h_{\ast}$). It follows that $\pi_1(\overline{H})$ is subdirect.
 \end{proof}

\section{Examples from maps onto a product of two elliptic curves}
\label{secExCombined}

In the light of Theorem \ref{thmExsGenClass}(3) and (4) one wonders if there is a relation between the finiteness properties of an irreducible full subdirect K\"ahler group $G\leq \pi_1(S_1)\times \dots \times \pi_1(S_r)$ and the maximal dimension of a smooth embedded subvariety $X\stackrel{\iota}{\hookrightarrow} R_1\times \dots \times R_r$ of a direct product of $r$ Riemann surfaces $R_i$ with $\iota_{\ast}(\pi_1 (X))\cong G$. In this section we will show that this is not the case. More precisely, we will show that for $2\leq m \leq r-1$ and any choice of surface groups $\pi_1(S_i)$ there is a full irreducible K\"ahler subgroup $G\leq \pi_1(S_1)\times \dots \times \pi_1(S_r)$ of finiteness type $\mathcal{F}_m$ but not of type $\mathcal{F}_{m+1}$ which is the fundamental group of a complex submanifold of complex codimension two in a direct product $S_1\times \dots \times S_r$ of Riemann surfaces. We obtain our examples by combining the ideas of the construction in Section \ref{secExamples} with Addendum \ref{addFiltVerGenWeak} and Theorem \ref{thmFiltVerGen}. 

Note that codimension two is the best we can hope for with our construction. Indeed, if $E$ is an elliptic curve, let $h: S_{1}\times \dots \times S_{r}\to E$ be a holomorphic map and let $(z_1,\dots, z_r)\in S_{1}\times \dots \times S_{r}$ be a base point. Then, for $2\leq m\leq r$, the restriction $h|_{\left\{(z_1,\dots,z_{m-1})\right\} \times S_{m}\times \left\{(z_{m+1},\dots,z_r)\right\}}$ is either trivial or surjective and thus the image $h_{\ast}(1\times \dots \times 1\times  \pi_1 (S_{m})\times 1 \times \dots \times 1)\leq \pi_1 (E)$ is either trivial or a finite index subgroup. Therefore, if $h$ induces a short exact sequence on fundamental groups, then $\ker (h_{\ast})$ is a product of finitely many surface groups and a group of the form constructed in \cite[Theorem 1.1]{Llo-16-II}. In particular, $\ker (h_{\ast})$ is not irreducible unless it is of type $\mathcal{F}_{r-1}$ and not of type $\mathcal{F}_r$.

As before, we choose elements $v_1,\dots, v_r \in \ZZ^2$, as well as ramified covers $\alpha_i: S_{i}\to E$ of an elliptic curve $E$, and define a holomorphic map 
\begin{equation}
h=\sum_{i=1}^r v_i\cdot \alpha_i : S_{1}\times \dots \times S_{r}\to E^{\times 2}.
\label{eqdefh}
\end{equation} 
However, we now vary our choice of the $v_i$'s to obtain examples with the desired properties. More precisely, for $r\geq 4$, $m\geq 1$ and $r-m\geq 3$, we choose $v_1,\dots, v_m=\left( \begin{array}{c} 1\\ 0 \end{array} \right)$, $v_{m+1}=\left(\begin{array}{c} 0\\ 1 \end{array}\right)$, $v_{m+2}=\left(\begin{array}{c} 1\\ 1\end{array}\right)$ and $\left\{ v_{m+3}, \dots, v_r \right\}\subset \ZZ^2$ such that $\left\{ v_i, v_j\right\}$ are linearly independent for $m\leq i < j \leq r$. For the remainder of this section we denote $\mathcal{C}:=\left\{v_1,\dots,v_r\right\}$ and $h$ the holomorphic map in \eqref{eqdefh} defined by this choice of $\mathcal{C}$. 

The main result of this section is:
\begin{theorem}
 Assume that $\mathcal{C} \subset \ZZ^2$ and $h$ are as defined above and that $\alpha_m$, $\alpha_{m+1}$ and $\alpha_{m+2}$ are surjective on fundamental groups. 
 
 Then $h$ satisfies the hypotheses of Addendum \ref{addFiltVerGenWeak} and Theorem \ref{thmFiltVerGen}, and $\overline{H}$ has the following properties:
 \begin{enumerate}
 \item $\overline{H}\subset S_1\times \dots \times S_r$ is a K\"ahler (projective) subvariety of codimension 2;
  \item the inclusion $\overline{H}\hookrightarrow S_1\times \dots \times S_r$ induces an embedding $\pi_1 (\overline{H})\leq \pi_1 (S_1)\times  \cdots \times \pi_1 (S_r)$ as an irreducible full subdirect product with $\pi_1(\overline{H})=\ker (h_{\ast})$; and
  \item $\pi_1 (\overline{H})$ is of finiteness type $\mathcal{F}_{r-m-1}$, but not of type $\mathcal{F}_{r-m}$.
 \end{enumerate} 
 Moreover, $\pi_i (\overline{H}) =\left\{0\right\}$ for $2\leq i \leq r-m-2$.

 \label{thmExtendedRange}
\end{theorem}

As in Section \ref{secExamples} we denote by $\pi_l : E^{\times 2}\to \left\{ 0 \right\} \times E^{\times l}$ the projection onto the last $l$ factors and by $h_l=\pi_l\circ h$ the composition. The map $h_l$ factors as $h_l = f_l \circ g_l$ for $1\leq l \leq 2$ with
\[
 f_2=\left(v_1\mid \dots \mid v_r \right) \circ \left(\alpha_1, \dots, \alpha_r \right) : S_{1}\times \dots \times S_{r}\to  E^{\times 2},
\]
\[
 g_2= S_{1}\times \dots \times S_{r} \to S_{1}\times \dots \times S_{r},
\]
\[
 f_1=\pi_1 \circ \left(v_{m+1}\mid \dots \mid v_r \right) \circ \left(\alpha_{m+1}, \dots, \alpha_r \right) : S_{m+1}\times \dots \times S_{r}\to  \left\{0\right\} \times E,
\]
and
\[
 g_1= S_{1}\times \dots \times S_{r} \to S_{m+1}\times \dots \times S_{r},
\]
where $g_1$ is the canonical projection with fibre $F_1=S_{1}\times \dots \times S_{m}$ a product of closed hyperbolic surfaces.

The main difference between the proof of Theorem \ref{thmExtendedRange} and the one of Theorem \ref{thmExsGenClass} is that the map $f_2$ does not have isolated singularities. However, by Addendum \ref{addFiltVerGenWeak} it suffices to show that the corestriction $f_2|_{f_2^{-1}(E\times \left\{e\right\})} : f_2^{-1}(E\times \left\{e\right\}) \to E\times \left\{ e \right\}$ has isolated singularities for $e\in E$ to obtain the consequences of Theorem \ref{thmFiltVerGen}. We thus start with two preliminary results that show that the $h_l$ do indeed satisfy the conditions of Addendum \ref{addFiltVerGenWeak}.

\begin{lemma}
\label{lemConnGenExsExt}
 Under the assumptions of Theorem \ref{thmExtendedRange}, the maps $h_i$, $f_i$ and $g_i$, $i=1,2$, have connected fibres and the fibres of the $g_i$ are aspherical.
\end{lemma}
\begin{proof}
 Since the maps $\alpha_m$ and $\alpha_{m+1}$ are surjective on fundamental groups, we see that by choice of $\mathcal{C}$, $r$ and $m$, the same argument as in the proof of Lemma \ref{propConnGenExs} shows connectedness of fibres. Moreover, the fibres of the $g_i$ are direct products of Riemann surfaces and thus aspherical.
\end{proof}

\begin{lemma}
 The map $f_1$ and the corestriction $f_2|_{f_2^{-1}(E\times \left\{e\right\})} : f_2^{-1}(E\times \left\{e\right\}) \to E\times \left\{ e \right\}$ have isolated singularities for all $e\in E$. 
 \label{lemIsSingExt}
\end{lemma}

\begin{proof}
 For the map $f_1$ this follows by the same argument as in the proof of Proposition \ref{propIsolSingGen}. The differential of $f_2$ is 
\begin{align*}
 Df_2 &= \left(v_1\cdot D\alpha_1 \dots v_r\cdot D\alpha_r\right)\\
      &= \left( \left( \begin{array}{c} 1\\ 0 \end{array}\right)\cdot D\alpha_1 \cdots \left(\begin{array}{c} 1\\ 0 \end{array} \right)\cdot D\alpha_m ~ v_{m+1} \cdot D\alpha_{m+1} ~\cdots ~ v_r \cdot D \alpha_r\right).
\end{align*}

Due to our assumptions on $\mathcal{C}$, a point $z=(z_1,\dots,z_r)\in S_{1}\times \dots \times S_{r}$ is a critical point of $f_2$ if and only if at least one of the following holds:
\begin{enumerate}
 \item $D\alpha_{m+1} (z_{m+1})=\dots = D\alpha_r (z_r)=0$;
 \item $D\alpha_1 (z_1)=\dots = D\alpha_m(z_m)= 0$ and at least $r-m-1$ of the $D\alpha_i(z_i)$ vanish for $i\geq m+1$.
\end{enumerate}

The locus of points satisfying (2) is a union of finitely many one-dimensional subvarieties of $S_{1} \times \dots \times S_{r}$ and it intersects each fibre of $f_2$ in at most finitely many points. Hence, these singularities are isolated for $f_2$ and thus for $f_2|_{f_2^{-1}(E\times \left\{e\right\})}$.

Observe that for the points $z\in f_2^{-1}(E\times \left\{e\right\})$ satisfying (1) but not (2), there exists at least one $i$ with $1\leq i\leq m$ such that $D\alpha_i(z_i)\neq 0$. The composition of $f_2|_{\left\{(z_1,\dots,z_{i-1})\right\} \times S_{i} \times \left\{ (z_{i+1},\dots, z_r\right\}}$ with the projection onto the second factor of $E^{\times 2}$ is constant. Thus, $\left\{(z_1,\dots,z_{i-1})\right\} \times S_{i} \times \left\{ (z_{i+1},\dots, z_r\right\} \subset f_2^{-1}(E\times \left\{e\right\})$ . In particular, $D\alpha_i(z_i)\neq 0$ implies that $D(f_2|_{f_2^{-1}(E\times \left\{e\right\})})(z)\neq 0$. Since $E\times \left\{e\right\}$ is one-dimensional, $z$ is not a critical point of the restriction $f_2|_{f_2^{-1}(E\times \left\{e\right\})}$. Hence, all critical points of $f_2|_{f_2^{-1}(E\times \left\{e\right\})}$ satisfy (2). 

It follows that $f_2|_{f_2^{-1}(E\times \left\{e\right\})}$ has isolated singularities.
\end{proof}

\begin{proof}[Proof of Theorem \ref{thmExtendedRange}]
We argue as in Section \ref{sec:Restrictions} that for a generic point $x_2\in E^{\times 2}$ the projections $x_i=\pi_i(x_2)$ are regular values of the $h_i$. Thus, Lemmas \ref{lemConnGenExsExt} and \ref{lemIsSingExt} imply that all conditions of Addendum \ref{addFiltVerGenWeak} are satisfied. Our conditions on $\mathcal{C}$ and the proof of Theorem \ref{thmFiltVerGen} show that (2), with the exception of irreducibility, the moreover part and (1) hold, and that $\pi_1 (\overline{H})$ is of type $\mathcal{F}_{r-m-1}$. 

Thus, we only need to prove that $\pi_1(\overline{H})$ is irreducible and not of type $\mathcal{F}_{r-m}$. Observe that $\phi=h_{\ast}$ is defined by the surjective homomorphism
 \[
  h_{\ast}(g_1,\cdots,g_r)=\phi(g_1,\cdots,g_r)=\sum_{i=1}^r v_i\cdot \alpha_i(g_i) 
\in \pi_1 (E^{\times 2}) = (\ZZ^2)^2\cong \ZZ^4
\]
for $(g_1,\cdots,g_r)\in \pi_1 (S_{1})\times \cdots \times \pi_1 (S_{r})$.

We first prove irreducibility. Assume that there is a finite index subgroup $H_1 \times H_2 \leq \ker (\phi)$. Since $\ker (\phi)$ is a full subdirect product (after possibly passing to finite index subgroups of the $\pi_1 (S_{i})$), we obtain that there is a partition of $\left\{1,\dots, r\right\}$ into two sets $\left\{i_1,\dots,i_t\right\}$ and $\left\{ i_{t+1},\dots, i_r\right\}$ such that $H_1 \leq \pi_1 (S_{i_1})\times \dots \times \pi_1 (S_{i_t})$ and $H_2 \leq \pi_1 (S_{i_{t+1}})\times \dots \times \pi_1(S_{i_r})$ (because non-trivial elements of surface groups have cyclic centralizers).

Thus, for every element $(x_{i_1},\dots,x_{i_r})\in H_1\times H_2\leq \Gamma_{i_1}\times \dots \times \Gamma_{i_r}$, we have $(x_{i_1},\dots,x_{i_t},1,\dots,1)\in H_1\times 1 \leq \ker (\phi)$ and $(1,\dots,1,x_{i_{t+1}},\dots,x_{i_r})\in 1 \times H_2 \leq \ker (\phi)$. By our assumptions on $\mathcal{C}$ at least one of $\left\{v_{i_1},\dots,v_{i_t}\right\}$ and $\left\{v_{i_{t+1}},\dots,v_{i_r}\right\}$ contains a pair of linearly independent vectors, say $\left\{v_{i_{t+1}},v_{i_{t+2}}\right\}$ are linearly independent.

Let $A\leq \ZZ^4 = (\ZZ^2)^2$ be the finite index subgroup with $$A= \mathrm{Im} \left( v_{i_{t+1}}\cdot\alpha_{i_{t+1},\ast} + v_{i_{t+2}}\cdot \alpha_{i_{t+2},\ast}\right).$$ Then for any element $x_{i_1}\in \Lambda_{i_1}$ of the finite index subgroup $\Lambda_{i_1}:=\left( v_{i_1} \cdot \alpha_{i_1 ,\ast}\right)^{-1}(A)\unlhd \Gamma_{i_1}$ there is $(x_{i_{t+1}},x_{i_{t+2}})\in \Gamma_{i_{t+1}}\times \Gamma_{i_{t+2}}$ with $(x_{i_1},1,\dots,1,x_{i_{t+1}},x_{i_{t+2}},1,\dots,1)\in \ker (\phi)$. Hence, the intersection $\ker (\phi) \cap \left(\Gamma_{i_1}\times \Gamma_{i_{t+1}}\times \Gamma_{i_{t+2}}\right)$ projects to a finite index subgroup of $\Gamma_{i_1}$. Since $H_1\times H_2\leq \ker (\phi)$ is a finite index subgroup, the same is true for the projection of the intersection $\left(H_1\times H_2\right) \cap \left(\Gamma_{i_1}\times \Gamma_{i_{t+1}}\times \Gamma_{i_{t+2}}\right)$ to $\Gamma_{i_1}$. Thus, $H_1\times 1 $ contains a finite index subgroup of $\Gamma_{i_1}$. This is impossible, because $\ker (\alpha_{i_1,\ast})\leq \Gamma_{i_1}$ has infinite index. It follows that $\ker (\phi)$ is irreducible.

Finally, observe that $v_1=\dots=v_m=\left(\begin{array}{c} 1\\ 0 \end{array}\right)$ implies that the image $\phi(\Gamma_1\times \dots \times \Gamma_m)\cong \ZZ^2 \leq \ZZ^4=\mathrm{Im} \phi$ is not a finite index subgroup. Thus, the group $\ker \phi$ is not of type $\mathcal{F}_{r-m}$, by Corollary \ref{propCoNilpFm} below.
\end{proof}

\begin{remark}
 Note that one could also apply the converse direction of Theorem \ref{thmKochlVSPk} to obtain that the groups constructed in Theorem \ref{thmExsGenClass} and Theorem \ref{thmExtendedRange} have finiteness type $\mathcal{F}_{r-k}$ (resp. $\mathcal{F}_{r-m-1}$). However, the proof we give actually provides the stronger result that in fact classifying spaces for our examples can be constructed from compact K\"ahler manifolds by attaching only cells of dimension larger than $r-k$ (resp. $r-m-1$).
\end{remark}

\begin{remark}
\label{remNotIsom}
Theorem \ref{thmExsGenClass}(4) and Theorem \ref{thmExtendedRange}(1) show that the examples obtained from Theorem \ref{thmExsGenClass} for $k\geq 3$ are not isomorphic to any of the examples obtained from Theorem \ref{thmExtendedRange}.
\end{remark}

\begin{remark}
 The construction described in this section generalises to epimorphisms to $\ZZ^{2k}$ for any $k\geq 2$. This produces irreducible coabelian K\"ahler subgroups of a direct product of $r\geq 3$ surface groups of type $\mathcal{F}_m$ and not $\mathcal{F}_{m+1}$ for $2\leq m \leq r-k$, which arise as fundamental groups of a codimension $k$ subvariety of a direct product of Riemann surfaces.  For simplicity of notation, we restricted ourselves to the case $k=2$ in our explicit computations.
\end{remark}

\section{Finiteness properties and irreducibility}
\label{secNotProd}

In this section we want to determine the precise finiteness properties of our examples and prove that they are irreducible.

 Let $H\leq G = G_1\times\dots \times G_r$ be a subgroup of a direct product of groups $G_1,\cdots, G_r$. For every $1\leq i_1 < \dots < i_k \leq r$ denote by $p_{i_1,\cdots,i_k}: G \rightarrow G_{i_1}\times \dots \times G_{i_k}$ the canonical projection. 
 
\noindent We say \cite{BriHowMilSho-13,Kuc-14} that the group $H$ \textit{virtually surjects onto $k$-tuples} (resp. \textit{surjects onto $k$-tuples}, resp. \textit{virtually surjects onto pairs (VSP)}) if for every $1\leq i_1 < \cdots < i_k\leq r$ the group $p_{i_1,\cdots, i_k}(H)$ has finite index in $G_{i_1}\times \dots \times G_{i_k}$ (resp. we have equality $p_{i_1,\cdots,i_k}(H)=G_{i_1}\times \dots \times G_{i_k}$, resp. $H$ virtually surjects onto $2$-tuples). 
 
 We further say that $H$ is \textit{coabelian of rank $k$} if there is an epimorphism $\phi: G_1\times \dots \times G_r\to \ZZ^k$ such that $H=\ker (\phi)$, and that $H$ is virtually coabelian of rank $k$ if there are finite index subgroups $H_0\leq H$ and $G_{i,0}\leq G_i$ such that $H_0\leq G_{1,0}\times \dots \times G_{r,0}$ is coabelian of rank $k$. It is not hard to see that the coabelian rank of $H$ is invariant under passing to finite index subgroups.
 
\begin{remark}
 For the remainder of the section we will be considering the finiteness properties and irreducibility of subgroups of direct products of surface groups. Note that the results in the literature that we refer to are stated in the more general context of non-abelian limit groups, which include surface groups and free groups. We want to emphasize at this point that in fact all results of this section, with the exception of Lemma \ref{lemNotProdNotFr}, remain true when replacing surface groups by non-abelian limit groups. The only change to the proofs is in Lemma 5.5, where we need to replace the fact that centralizers in surface groups are cyclic by the facts that non-abelian limit groups have trivial center and can not decompose as a direct product of two non-trivial groups.
\end{remark}

For subgroups of direct products of surface groups, a close relation between their finiteness properties and virtual surjection to $k$-tuples has been observed (see \cite{BriHowMilSho-13},\cite{Koc-10}, also \cite{Kuc-14}). In fact if a subgroup $H\leq G_1\times \cdots \times G_r$ of a direct product of finitely presented groups is subdirect then $H$ is finitely generated; and if it is VSP then $H$ is itself finitely presented \cite[Theorem A]{BriHowMilSho-13}. The converse is not true in general; it is true though if $G_1,\cdots, G_r$ are surface groups and $H$ is full subdirect \cite[Theorem D]{BriHowMilSho-13}.

More generally, it is conjectured \cite{Koc-10} that, for $G_1,\cdots, G_r$ surface groups (resp. non-abelian limit groups) and $H\leq G_1\times \cdots \times \times G_r$ a full subdirect product, the following are equivalent:
\begin{enumerate}
 \item $H$ is of type $\mathcal{F}_k$;
 \item $H$ virtually surjects onto $k$-tuples.
\end{enumerate}  

Kochloukova proved that (1) implies (2) and gave conditions under which (2) implies (1).

\begin{theorem}[{Kochloukova \cite[Theorem C]{Koc-10}}]
 For $r\geq 1$ let $\G_1, \dots , \G_r$ be surface groups, let $H\leq \G_1\times \dots \times \G_r$ be a full subdirect product, and let $2\leq k \leq r$. If $H$ is of type $\mathcal{F}_k$ then $H$ virtually surjects onto $k$-tuples. The converse is true if $H$ is virtually coabelian.
 \label{thmKochlVSPk}
\end{theorem}

Note that in Kochloukova's statement of Theorem \ref{thmKochlVSPk} the condition is that $H$ has the homological finiteness type $\rm{FP}_k(\mathds{Q})$. By \cite[Corollary E]{BriHowMilSho-13} this is however equivalent to type $\mathcal{F}_k$ for subgroups of direct products of surface groups. In general we only have that $\mathcal{F}_k$ implies $\rm{FP}_k(\mathds{Q})$ (see for instance \cite[Section 8.2]{Geo-08}).

We shall need the following auxiliary result which is a consequence of Theorem \ref{thmKochlVSPk}.

\begin{lemma}
 Let $G_1,\cdots, G_r$ be groups and $Q$ be a finitely generated abelian group. Let $\phi: G_1 \times \dots \times G_r \rightarrow Q$ be an epimorphism. 
 Assume that the subgroup $H=\ker (\phi) \leq G_1\times \dots \times G_r$ virtually surjects onto $m$-tuples (resp. surjects onto $m$-tuples), then the group $\phi( G_{i_1}\times \cdots \times G_{i_{r-m}})\leq Q$ is a finite index subgroup of $Q$ (resp. equal to $Q$) for all $1\leq i_1 < \dots < i_{r-m}\leq r$.
\label{lemSurjComp}
\end{lemma}

\begin{proof}
 Assume that $H$ virtually surjects onto $m$-tuples. Consider a product $ G_{i_1}\times \dots \times G_{i_{r-m}}$ of $r-m$ factors. We may assume that  $i_j=j$.
 
Let $g\in Q$ be an arbitrary element. By surjectivity of $\phi$ there exist elements $h_1\in G_1\times \cdots \times G_{r-m}$ and $h_2 \in G_{r-m+1}\times \cdots \times G_r$ such that $g=\phi(h_1)\cdot \phi(h_2)$. Since $H$ virtually surjects onto $m$-tuples there is $k\geq 1$ such that $h_2^k\in p_{r-m+1,\cdots,r}(H)$. Hence, there is $\overline{h}_1\in G_{1}\times \cdots \times G_{r-m}$ such that $\overline{h}_1 \cdot h_2^k\in H=\ker (\phi)$. In particular it follows that $\phi(h_2^k)=\phi((\overline{h}_1)^{-1})$. As a consequence we obtain that $g^k = \phi(h_1)^k \cdot \phi(h_2)^k = \phi(h_1)^k\cdot \phi((\overline{h}_1)^{-1})\in \phi(G_{1}\times \cdots \times G_{r-m})$.

We proved that the abelian group $Q/\phi(G_{1}\times \cdots \times G_{r-m})$ has the property that each of its elements is torsion. This implies that $Q/\phi(G_{1}\times \cdots \times G_{r-m})$ is finite and thus $\phi(G_{1}\times \cdots \times G_{r-m})$ is a finite index subgroup of $Q$.

The second part follows immediately, since we can choose $k=1$ in the above proof if H surjects onto $G_{r-m+1}\times \dots \times G_r$.
\end{proof}

\begin{corollary}
 Let $\phi: \G_1 \times \cdots \times \G_r\rightarrow Q$ be an epimorphism, where $\G _1,\cdots,\G_r$ are surface groups and $Q$ is a finitely generated abelian group. If $\ker (\phi)$ is a full subdirect product of type $\mathcal{F}_m$ then the image $\phi(\G_{i_1} \times \cdots \times \G_{i_{r-m}})\leq Q$ is a finite index subgroup of $Q$ for all $1\leq i_1< \cdots < i_{r-m}\leq r$.
 \label{propCoNilpFm}
\end{corollary}
\begin{proof}
This is a direct consequence of Lemma \ref{lemSurjComp} and Theorem \ref{thmKochlVSPk}.
\end{proof}

As another consequence of Theorem \ref{thmKochlVSPk}, we obtain: 

\begin{lemma}
 Let $G\leq \G_1\times \dots \times \G_r$ be a full subdirect product of $r$ surface groups $\G_i$ of type $\mathcal{F}_m$ which is virtually a product $H_1\times H_2$. Then, after possibly reordering factors, there is $1\leq s \leq r$ such that $H_1\leq \G_1 \times \dots \times \G_s$, $H_2\leq \G_{s+1}\times \dots \times \G_r$, and one of the following holds:
 \begin{enumerate}
  \item $H_1$ is virtually a direct product of $s$ surface groups; or
  \item $G$ virtually surjects onto at least one $2m$-tuple.
 \end{enumerate}
 \label{lemmNotProdKochl}
\end{lemma}
\begin{proof}
 Centralizers of non-trivial elements in surface groups are infinite cyclic. Hence, subdirectness of $G$ implies that, after reordering the $\G_i$, there is $1\leq s \leq r$ such that $H_1\leq \G_1 \times \dots \times \G_s$ and $H_2\leq \G_{s+1}\times \dots \times \G_r$. Theorem \ref{thmKochlVSPk} implies that $G$ and thus $H_1\times H_2$ virtually surjects onto $m$-tuples in $\G_1 \times \dots \times \G_r$. 
 
 If either $s\leq m$ or $r-s\leq m$, say $s\leq m$, then $H_1 \times H_2$ surjects onto a finite index subgroup of $\G_1 \times \dots \times \G_s$. However, this projection is isomorphic to $H_1$. Thus, $H_1$ is a finite index subgroup of $\G_1 \times \dots \times \G_s$ and therefore $H_1$ is virtually a direct product of $s$ surface groups. If $r-s\leq m$, then a further reordering of factors allows us to exchange the roles of $H_1$ and $H_2$ (and $s$ and $r-s$), thus leading to the same conclusion.
 
 Now assume that $s, ~ r-m > m$. In this case $H_1 \times H_2 $ virtually surjects onto $m$-tuples in $\G_1 \times \dots \times \G_s$ (resp. $\G_{s+1}\times \dots \times \G_{r}$). Factoring the projection onto such an $m$-tuple through the projection onto $\G_1 \times \dots \times \G_s$ (resp. $\G_{s+1}\times \dots \times \G_r$), we see that $H_1\leq \G_1 \times \dots \times \G_s$ (resp. $H_2\leq \G_{s+1}\times \dots \times \G_r$) also virtually surjects onto $m$-tuples. Thus, $H_1\times H_2$ virtually surjects onto at least one $2m$-tuple and the same holds for $G$. 
\end{proof}

\begin{remark}
Note that Lemma \ref{lemmNotProdKochl} can also be applied to full subgroups $G\leq \G_1\times \dots \times \G_r$, after replacing the $\G_i$ by the projections of $G$ to $\G_i$ and $G$ by the quotient $G/Z(G)$ by the center $Z(G)$, since $G/Z(G)$ and $G$ have the same finiteness type by \cite[Proposition 2.7]{Bie-81}.
\end{remark}

We shall also need the following result by Kuckuck.

\begin{proposition}[{\cite[Corollary 3.6]{Kuc-14}}]
 Let $G\leq \G_1 \times \cdots \times \G_r$ be a full subdirect product of a direct product of $r$ surface groups $\G_i$, $1\leq i \leq r$. If $G$ virtually surjects onto $m$ tuples for $m> \frac{r}{2}$ then $G$ is virtually coabelian. In particular, $G$ is virtually coabelian if $G$ is of type $\mathcal{F}_m$. 
 
 More precisely, we have that in either case there exist finite index subgroups $\S_i\leq \G_i$, a free abelian group $A$ and a homomorphism
 \[
 \phi: \S_1\times \cdots \times \S_r \rightarrow A
 \]
such that $\ker (\phi)\leq G$ is a finite index subgroup.
\label{thmCoNilpFm}
\end{proposition}

We will require the following consequence of Theorem \ref{thmKochlVSPk} and Proposition \ref{thmCoNilpFm}:
\begin{corollary}
\label{corFinPropsProjFactors}
 Let $r\geq 1$ and let $G\leq \G_1 \times \cdots \times \G_r$ be a full subdirect product of surface groups $\G_i$, $1\leq i \leq r$. Assume that $G$ is of type $\mathcal{F}_m$ with $m\geq 0$. For $k\geq 0$ with $m>\frac{k}{2}$ and $1\leq i_1 < \dots < i_k\leq r$ the projection $p_{i_1,\cdots,i_k}(G)\leq \G_{i_1}\times \dots \times \G_{i_k}$ is of type $\mathcal{F}_m$.
\end{corollary}
\begin{proof}
 By Theorem \ref{thmKochlVSPk} the group $G\leq \G_1 \times \cdots \times \G_r$ virtually projects onto $m$-tuples. Hence, the projection $\overline{G}:=p_{i_1,\cdots,i_k}(G)\leq \G_{i_1}\times \cdots \times \G_{i_k}$ is full subdirect and virtually surjects onto $m$-tuples with $m> \frac{k}{2}$. By Proposition \ref{thmCoNilpFm} $\overline{G}$ is virtually coabelian in $\G_{i_1}\times \cdots \times \G_{i_k}$. Hence, the subgroup $\overline{G}\leq \G_{i_1}\times \cdots \times \G_{i_k}$ is full subdirect, virtually coabelian, and virtually projects onto $m$-tuples. The converse direction of Theorem \ref{thmKochlVSPk} then implies that $\overline{G}$ is of type $\mathcal{F}_m$.
\end{proof}

As a consequence of the results in this section we can determine the precise finiteness properties and irreducibility of the groups arising from our construction in Theorem \ref{thmExsGenClass}.
\begin{lemma}
 Under the assumptions of Theorem \ref{thmExsGenClass} and with the same notation, let  $\phi=h_{\ast}: \pi_1 (S_{1})\times \cdots \times \pi_1 (S_{r})\rightarrow \pi_1 (E^{\times k})$ be the induced epimorphism on fundamental groups. 
 
 Then $\ker (\phi) \cong \pi_1 (\overline{H})$ is irreducible of type $\mathcal{F}_{r-k}$ and not of type $\mathcal{F}_{r-k+1}$. Moreover, there is no subvariety $X\stackrel{\iota}{\hookrightarrow}R_1 \times \dots \times R_r$ of dimension $\geq r-k+1$ in a direct product of $r$ Riemann surfaces $R_i$ of genus $\geq 2$ such that $\iota_{\ast}(\pi_1 (X))\cong\pi_1 (\overline{H})$.
\label{lemNotProdNotFr}
\end{lemma}
\begin{proof}
 Since $\overline{H}$ is compact with $\pi_i(\overline{H})=\left\{1\right\}$ for $2\leq i \leq r-k$, the group $\pi_1 (\overline{H})$ admits a classifying space with finitely many cells in dimension less than or equal to $r-k$. Thus, $\pi_1(\overline{H})$ is of type $\mathcal{F}_{r-k}$. 
 
 As in the proof of Theorem \ref{thmExtendedRange} we have
 \[
  \phi(g_1,\cdots,g_r)=\sum_{i=1}^r v_i\cdot \alpha_i(g_i) \in (\pi_1( E))^{\times k}\cong (\ZZ^2)^{k}\cong \ZZ^{2k}
 \]
 for $(g_1,\cdots,g_r)\in \pi_1 (S_{1})\times \cdots \times \pi_1 (S_{r})$.
 
 Since the maps $\alpha_i$ are finite sheeted branched coverings, the image $\alpha_{i,\ast}(\pi_1(S_{i}))\leq \pi_1 (E)$ is a finite index subgroup for $1\leq i \leq r$. The assumption that the $v_i$ satisfy property (P') implies that the image $\phi(\pi_1 (S_{i_1})\times \cdots \times \pi_1 (S_{i_k}))\leq \pi_1 (E^{\times k})$ of any $k$ factors is a finite index subgroup of $\pi_1 (E^{\times k})\cong \ZZ^{2k}$, $1\leq i_1 < \cdots < i_k\leq r$.

 Since we have $r\geq k+2$ factors and any $k$ factors map to a finite index subgroup of $\pi_1 (E^{\times k})$, the kernel of 
\[ 
 \phi_0=\phi|_{\S_1\times \cdots \times \S_r}: \S_1 \times \cdots \times \S_r\rightarrow \pi_1 (E^{\times k})
\]
is subdirect, after passing to finite index subgroups $\S_i\leq \pi_1 (S_{\g_i})$. Note that the image $\rm{Im} (\phi_0)\leq \pi_1 (E^{\times k})$ is a finite index subgroup, thus isomorphic to $\ZZ^{2k}$, and that $\ker (\phi_0) \leq \ker (\phi)$ is a finite index subgroup. The intersection $\S_i \cap \ker (\phi_0)\unlhd \S_i$ is a non-trivial normal subgroup of infinite index in $\S_i$, since $\phi(\S_i)\cong \ZZ^2$. Thus, $\ker (\phi_0)$ is a full subdirect product of $\S_1 \times \cdots \times \S_r$.
 
 Since the image of the restriction of $\phi_0$ to any factor $\S_i$ is isomorphic to $\ZZ^2$, the image of the restriction of $\phi$ to any $k-1$ factors $\S_{i_1}\times \cdots \times \S_{i_{k-1}}$ ($1\leq i_1 < \cdots < i_{k-1}$) is isomorphic to $\ZZ^{2(k-1)}$ (by the same argument as for $k$ factors). In particular, $\phi(\S_{i_1}\times \cdots \times \S_{i_{k-1}})$ is not a finite index subgroup of ${\rm{Im}} (\phi_0)\cong \ZZ^{2k}$. By Corollary \ref{propCoNilpFm}, $\ker (\phi_0)$ and therefore its finite extension $\ker (\phi) \geq \ker (\phi_0)$ cannot be of type $\mathcal{F}_{r-k+1}$. 
 
 Assume that there is a finite index subgroup $H_1 \times H_2 \leq \ker (\phi)$ which is a product of two non-trivial groups $H_1$ and $H_2$. By Lemma \ref{lemmNotProdKochl} we may assume that (after reordering factors) $H_1 \leq \pi_1 (S_{1}) \times \cdots \times \pi_1 (S_{s})$, $H_2 \leq \pi_1 (S_{s+1})\times \cdots \pi_1 (S_{r})$, for some $1\leq s \leq r-1$. Assume that $H_1$ (or $H_2$) is virtually a direct product of surface subgroups $\Lambda_i \leq \pi_1 (S_{i})$, $1\leq i \leq s$. In particular, the $\Lambda_i$ must be finite index subgroups of the $\pi_1 (S_{i})$. This contradicts that the restriction of $\phi$ to any finite index subgroup of $\pi_1 (S_{i})$ has infinite image. Hence, by Lemma \ref{lemmNotProdKochl}, $\ker (\phi)$ virtually surjects onto at least one $2(r-k)$-tuple. However, the genericity condition (P') satisfied by $\mathcal{C}$ implies that $\ker (\phi)$ does not virtually surject onto any $(r-k+1)$-tuple (the argument is the same as in the proof of Lemma \ref{lemSurjComp}). In particular, $\ker (\phi)$ is irreducible.
 
 Finally, assume that there are closed Riemann surfaces $R_i$ and a smooth subvariety $X\stackrel{\iota}{\hookrightarrow} R_1\times \dots \times R_r$ of dimension $\geq r-k+1$ with $\iota_{\ast}(\pi_1(X))\cong\pi_1 (\overline{H})$. The group $\iota_{\ast}(\pi_1(X))\leq \pi_1(R_1)\times \dots \times \pi_1(R_r)$ is full, since by assumption it must contain $\ZZ^r$ as a subgroup. After replacing the $R_i$ by finite covers, we may also assume that $\iota_{\ast}(\pi_1(X))$ is a subdirect product. Thus, \cite[Theorem C(3)]{BriHowMilSho-13} implies that, after possibly reordering factors, the isomorphism $\iota_{\ast}(\pi_1(X))\cong\pi_1 (\overline{H})$ is induced by an isomorphism $\nu: \pi_1(R_1)\times \dots \times \pi_1(R_r)\to \pi_1(S_1)\times \dots \times \pi_1(S_r)$ with $\nu(\pi_1(R_i))=\pi_1(S_i)$. In particular, we deduce, from the above observation that $\pi_1(\overline{H})$ does not virtually surject onto any $(r-k+1)$-tuple, that the same holds for $\iota_{\ast}(\pi_1 (X))$.
 
 On the other hand the fact that $X$ has dimension at least $r-k+1$ implies that the projection of $X$ to at least one $(r-k+1)$-tuple of direct factors must be surjective. After possibly reordering factors we may assume that $\iota(X)$ surjects onto $R_1\times \dots \times R_{r-k+1}$. It follows that $p_{1,\dots,r-k+1}(\iota_{\ast}(\pi_1(X))) \leq \pi_1(R_1)\times \dots \times \pi_1(R_{r-k+1})$ is a finite index subgroup, a contradiction. This completes the proof.
\end{proof}

\section{Restrictions on coabelian K\"ahler groups}
\label{secResCoabKGs}

All non-trivial examples of K\"ahler subgroups of direct products of surface groups constructed so far, are obtained as kernels of maps from a direct product of surface groups to a free abelian group. Hence, a natural special case of Delzant and Gromov's question is the following question.

\begin{question}
 Let $S_{1},\dots, S_{r}$ be closed hyperbolic Riemann surfaces, let $k\in \ZZ$ and let $\phi: \pi_1 (S_{1})\times\dots \times \pi_1 (S_{r})\to \ZZ^k$ be an epimorphism. When is $\ker (\phi)$ a K\"ahler group?
\end{question}

We will show that for $\ker (\phi)$ to be K\"ahler, $k$ must be even. This will follow from a study of the more general question of finding constraints on homomorphisms
\[
 \psi: G \rightarrow \pi_1 (S_{1}) \times \dots \times \pi_1 (S_{k})
\]
from a K\"ahler group $G$ to a direct product of surface groups.

For a compact K\"ahler manifold $X$ we denote by $a_X: X\to A(X)$ the Albanese map to its Albanese variety $A(X)$. By the universal property of the Albanese variety, every holomorphic map $f: X\to Y$ between compact K\"ahler manifolds $X$ and $Y$ induces a commutative diagram of holomorphic maps
\[
\xymatrix{ X\ar[r]^{f}\ar[d]_{a_X}& Y\ar[d]^{a_Y}\\ A(X) \ar[r]^{f_{\mathrm{alb}}}& A(Y).}
\]

The following result shows that this commutative diagram of holomorphic maps provides strong restrictions on coabelian subgroups of K\"ahler groups which are the image of a holomorphic map.

\begin{lemma}
Let $X,Y$ be compact K\"ahler manifolds and let $f:X\rightarrow Y$ be a holomorphic map. Then the images, kernels and cokernels of the induced maps
\[
 f_{\ast} : H_1(\pi_1 (X),\ZZ)=(\pi_1 (X))_{ab} \rightarrow H_1(\pi_1 (Y), \ZZ)=(\pi_1 (Y))_{ab},
\]
\[
f^{\ast}: H^1(\pi_1 (X),\ZZ) \rightarrow H^1(\pi_1(Y),\ZZ)
\]
have even rank.

If, moreover, $f_{\ast}(\pi_1 (X)) \leq \pi_1 (Y)$ is the kernel of an epimorphism $\psi: \pi_1 (Y) \to \ZZ^l$ then $f_{\ast}(\pi_1(X))$ is holomorphically coabelian. More precisely, $f_{\ast}(\pi_1(X))$ is the kernel of the homomorphism $h_{\ast}:\pi_1(Y)\to \pi_1(B)=\ZZ^l$ induced by the holomorphic map $h: Y\to A(Y)/f_{\mathrm{alb}}(A_X)=:B$. 
\label{lemHodgeHol}
\end{lemma}

\begin{proof}
The first part is an easy consequence of the observation that
\begin{equation}
f_{\mathrm{alb},\ast}=f_{\ast,tf}: H_1(A(X),\ZZ)=H_1(X,\ZZ)_{tf} \to H_1(X,\ZZ)_{tf}=H_1(A(Y),\ZZ),
\label{eqnAlbHom}
\end{equation}
where for a finitely generated abelian group $A$ we denote by $A_{tf}=A/\left\{{\mathrm{torsion}}\right\}$ its maximal torsion-free quotient (see also \cite[Lemma 2.1]{Ara-11}).

Now assume that $G:=f_{\ast}(\pi_1 (X)) =\ker (\psi)$ for an epimorphism $\psi: \pi_1 (Y)\to \ZZ^l$. Then we obtain an exact sequence
 \[
    \pi_1(X) \stackrel{f_{\ast}}{\rightarrow} \pi_1 (Y) \stackrel{\psi}{\rightarrow} \ZZ^l\rightarrow 1.
 \]
By right exactness of abelianization, this induces an exact sequence
 \[
  (\pi_1(X))_{ab}\stackrel{f_{\ast}}{\rightarrow} (\pi_1(Y))_{ab}\rightarrow \ZZ^l \rightarrow 1.
 \]
Since $\ZZ^l$ is torsion-free, this sequence induces an exact sequence on maximal torsion-free quotients which forms the lower row of the following commutative diagram
\begin{equation}
 \xymatrix{\pi_1(X) \ar[r]^{f_{\ast}}\ar[d]& \pi_1 (Y) \ar[r]^{\psi}\ar[d]& \ZZ^l\ar[r]\ar[d]^{=}&1\\
 (\pi_1(X))_{ab,tf}\ar[r]^{f_{\ast,tf}}& (\pi_1(Y))_{ab,tf}\ar[r]& \ZZ^l \ar[r]& 1.}
\label{eqnNewCoab2}
\end{equation}
It follows that $\ZZ^l$ is the cokernel of $f_{\ast,tf}=f_{\mathrm{alb},\ast}$. 
 
On the other hand we obtain a commutative diagram of holomorphic maps
\[
\xymatrix{ X\ar[r]^{f}\ar[d]_{a_X}& Y\ar[d]^{a_Y}\ar[rd]^h\\ A(X) \ar[r]^{f_{\mathrm{alb}}}& A(Y)\ar[r]& A(Y)/f_{\mathrm{alb}}(A(X))=B,}
\]
denoting by $h$ the induced diagonal map. Combining this with \eqref{eqnAlbHom} and \eqref{eqnNewCoab2}, we deduce that there is a commutative diagram
\begin{equation}
\xymatrix{ \pi_1 (X) \ar[r]^{f_{\ast}}\ar[d]& \pi_1(Y)\ar[r]\ar[d]\ar[rd]& \ZZ^l \ar[r]\ar[d]& 1\\
\pi_1(A(X))\ar[r]^{f_{{\mathrm{alb}},\ast}}& \pi_1(A(Y))\ar[r]\ar[ur]& \pi_1\left(B\right),&}
\label{eqnNewCoab3}
\end{equation} 
where all vertical and diagonal maps are epimorphisms. 

For dimension reasons it follows from the short exact sequences in \eqref{eqnNewCoab2} and the definition of $B$ that $l=2\cdot {\mathrm{dim}}_{\C}(B)$. We deduce that the vertical homomorphism $\ZZ^l \to \pi_1(B)$ in \eqref{eqnNewCoab3} is an isomorphism and thus the result follows.
\end{proof}

Lemma \ref{lemHodgeHol} allows us to strengthen our results with respect to a previous version of this work; it is based on a suggestion by the referee, for which we are very grateful. Before proceeding to state and prove the main result of this section, we want to recall some results from the literature that we shall need.

For $g\geq 0$ and $\mm=(m_1,\dots,m_s)$, we denote by $S_{g,\mm}$ the closed orbisurface of genus $g$ with $s\geq 0$ cone points $D=\left\{p_1,\dots,p_s\right\}$ of multiplicities $m_i\geq 2$. Its \textit{orbifold fundamental group} is the group
\[
 \G_{g,\mm}:=\pi_1^{orb} (S_{g,\mm}) = \pi_1 (S_g \setminus D)/ \left\langle\left\langle \gamma_i ^{m_i} \mid i=1,\dots,s \right\rangle\right\rangle
\]
for $\gamma_i$, $1\leq i \leq r$, a loop bounding a small disc around $p_i$. We say that $S_{g,\mm}$ is equipped with a complex structure if the underlying surface $S_g$ is equipped with a complex structure; in this case we call $S_{g,\mm}$ a \textit{Riemann orbisurface}. For a complex manifold $X$, a map $f: X\rightarrow S_{g,\mm}$ is called holomorphic if for each $p_i\in D$ there is a disc neighbourhood $U_i$ in which $f$ factors through a holomorphic map to the $m_i$-fold branched cover of $U_i$ in $p_i$.

Throughout the remainder of this paper all orbisurfaces will be assumed to be closed and hyperbolic. We will use this assumption without further mention and we will refer to the orbifold fundamental group of a closed hyperbolic orbisurface as an \textit{orbisurface group}.

\begin{theorem}
 Let $X$ be a compact K\"ahler manifold, let $G=\pi_1 (X)$, and let $\phi: G \rightarrow \pi_1^{orb} (S_{g,\mm})$ be an epimorphism. Then $\phi$ factors through an epimorphism $\psi: G\rightarrow \pi_1^{orb} (S_{h,\nn})$ with finitely generated kernel, which is induced by a holomorphic map $X \to S_{h,\nn}$.
 
 Moreover, $\phi$ has finitely generated kernel if and only if $\phi$ is induced by a holomorphic map $f: X \to S_{g,\mm}$ with connected fibres (for a suitable complex structure on $S_{g,\mm}$), such that the critical values of $f$ are the cone points $p_i$ and the multiplicity of the singular fibre over $p_i$ is $m_i$.
 \label{lemSiuBeauCat}
\end{theorem}

The surface group case of the first part of Theorem \ref{lemSiuBeauCat} is due to Siu \cite{Siu-87} and Beauville \cite{Bea-91}. This version of Theorem \ref{lemSiuBeauCat} is proved in \cite{Del-16}. A first explicit version of this result can be found in \cite{Cat-03}; however, it was probably known much earlier (see discussion in \cite{Kot-12}). We will require the following additional observation (see for instance \cite[Remark 2.13]{ABCKT-95}).
\begin{remark}
 Let $G$ be a K\"ahler group and let $\phi: G \to F_r$ be an epimorphism onto a  finitely generated free group. Then $\phi$ factors through an epimorphism $\psi: G\to \G_{g,\mm}$ with finitely generated kernel.
 \label{rmkSiuBeauCat}
\end{remark}

\begin{theorem}
Let $\G_{g,\mm}$ be an orbisurface group, let $A$ be any group and let $G\leq \G_{g,\mm} \times A$. Assume that $G$ is finitely presented and that the intersection $\G_{g,\mm} \cap G$ is infinite. Then $G\cap A$ is finitely generated.
\label{thmBriMil}
\end{theorem}
\begin{proof}
 If $\G_{g,\mm}$ is a non-abelian surface group this is \cite[Theorem 4.6]{BriMil-09}. The general case follows after passing to a finite index subgroup $G_0=G\cap \left(\G_h\times A\right)$ for $\G_h\leq \G_{g,\mm}$ a finite index surface subgroup, since $G_0\cap A = G\cap A$.  
\end{proof}

\begin{lemma}
 Let $G$ be a K\"ahler group and let $\phi: G\rightarrow \overline{G}\leq \G_{g_1,\mm_1}\times \dots \times \G_{g_r,\mm_r}$ be an epimorphism such that the projections $p_i\circ \phi: G \rightarrow \G_{g_i,\mm_i}$ to factors have finitely generated kernel. Then the projection $p_i(\overline{G})\leq \G_{g_i,\mm_i}$ is either virtually cyclic or of finite index.
 \label{lemSurCyc}
\end{lemma}

\begin{proof}
 After possibly passing to finite index subgroups of $\overline{G}$, $G$ and the $\G_{g_i,\mm_i}$, we may assume that $\overline{G}\leq \G_{g_1}\times \dots \times \G_{g_r}$ is a subgroup of a direct product of surface groups. Assume that $p_i(\overline{G})\leq \G_{g_i}$ is not cyclic. Since $p_i(\overline{G})$ is finitely generated, it is either finitely generated free or a surface group. Surface subgroups of $\G_{g_i}$ are precisely the finite index subgroups. Thus, assume that $F_r=p_i(\overline{G})$ is a non-abelian finitely generated free group.
 
Remark \ref{rmkSiuBeauCat} implies that there is a commutative diagram
\[
\xymatrix{ G\ar[r] ^{\psi}\ar[rd]_{p_i\circ \phi} & \pi_1^{orb}(S_{h,\mm})\ar[d]^{\theta}\\ & F_r=p_i(\overline{G}),}
\]
where $\psi$ is an epimorphism with finitely generated kernel for some orbisurface $S_{h,\mm}$. 

Since $p_i\circ \phi$ has finitely generated kernel the same is true for the induced epimorphism $\theta$. However, the only infinite index finitely generated normal subgroup of a closed hyperbolic orbisurface is the trivial group. This implies that $p_i(\overline{G})\cong \pi_1^{orb} (S_{h,\mm})$, a contradiction.
\end{proof}

As a direct consequence of Lemma \ref{lemSurCyc}, we obtain:

\begin{corollary}
 Every K\"ahler subgroup $G$ of a direct product of finitely many orbisurface groups is of the form $G\cong A \times G_0$ with $A$ virtually abelian and $G_0$ a full subdirect product of finitely many orbisurface groups.
\end{corollary}

\begin{proposition}
Let $G=\pi_1 (X)$ and assume that there is a homomorphism $\phi: G\rightarrow \G_{g_1,\mm_1}\times \dots \times \G_{g_r,\mm_r}$ that is induced by a holomorphic map $f: X \rightarrow S_{g_1,\mm_1}\times \dots \times S_{g_r,\mm_r}$. Assume further that the image $\overline{G}=\phi(G)\leq \G_{g_1,\mm_1}\times \dots \times \G_{g_r,\mm_r}$ is full and of type $\mathcal{F}_k$ for $k>\frac{r}{2}$. Then $\overline{G} \leq \G_{g_1,\mm_1}\times \dots \times \G_{g_r,\mm_r}$ is virtually holomorphically coabelian.
\label{propNewCoab}
\end{proposition}
\begin{proof}
 Since the composition $X\rightarrow S_{g_i,\mm_i}$ of $f$ with the projection to $S_{g_i,\mm_i}$ is holomorphic, it is either surjective or constant. The latter can not happen, since we assumed that $\overline{G}\leq \G_{g_1,\mm_1}\times \dots \times \G_{g_r,\mm_r}$ is full. Hence, the projection $p_i(\overline{G})\leq \G_{g_i,\mm_i}$ is a finite index subgroup. It follows that $\overline{G}\leq p_1(\overline{G})\times \dots \times p_r(\overline{G})$ is a full subdirect product of orbisurface groups.
 
By passing to finite index surface subgroups of the $p_i(\overline{G})$ and then applying Proposition \ref{thmCoNilpFm}, we deduce that there exist $l\geq 0$, finite index surface subgroups $\G_{h_i}\leq p_i(\overline{G})$ and $\overline{G}_0\leq \overline{G}$, and an epimorphism $\psi: \G_{h_1}\times \dots \times \G_{h_r}\rightarrow \ZZ^l$ such that $\overline{G}_0=\ker (\psi)$. 

Let $q: X_0\rightarrow X$ be the holomorphic cover corresponding to $\overline{G}_0$ and let $q_i: S_{h_i}\rightarrow S_{g_i,\mm_i}$ be the holomorphic covers corresponding to $\G_{h_i}$. Since $(f\circ q)_{\ast} (\pi_1 (X_0))\leq \G_{h_1}\times \dots \times \G_{h_r}$, there is a lift $g: X_0 \rightarrow S_{h_1}\times \dots \times S_{h_r}$ making the diagram
\[
\xymatrix{ & S_{h_1}\times \dots \times S_{h_r}\ar[d]^{(q_1,\dots,q_r)}\\
		   X_0 \ar[ur]^g \ar[r]_(.25){f\circ q} & S_{g_1,\mm_1}\times \dots \times S_{g_r,\mm_r}}
\]
commutative. Considering local charts we see that the map $g$ is holomorphic. It follows from Lemma \ref{lemHodgeHol} that $\overline{G}_0\leq \G_{h_1}\times \dots \times \G_{h_r}$ is holomorphically coabelian.
\end{proof}

Theorem \ref{thmNewCoab} will be a consequence of the following more general result for orbisurfaces and its proof. 

\begin{theorem}
Let $G=\pi_1 (X)$ with $X$ compact K\"ahler and let $\phi: G \rightarrow \overline{G}$ be a surjective homomorphisms onto a subgroup $\overline{G}\leq \G_{g_1,\mm_1}\times \dots \times \G_{g_r,\mm_r}$. Assume that $\phi$ has finitely generated kernel and that $\overline{G}$ is full and of type $\mathcal{F}_m$ for $m\geq 2$.
 
 Then, after reordering factors, there is $s\geq 0$, such that, for any $k< 2m$ and any $1\leq i_1 < \dots < i_k\leq s$, the projection $p_{i_1,\dots,i_k}(\overline{G})\leq \G_{g_{i_1},\mm_{i_1}}\times \dots \times \G_{g_{i_k},\mm_{i_k}}$ is virtually holomorphically coabelian. Furthermore, $\overline{G} \cap \left(\G_{g_{s+1},\mm_{s+1}}\times \dots \times \G_{g_r,\mm_{r}}\right)\leq p_{s+1,\dots,r}(\overline{G})$ is a virtually abelian finite index subgroup. 

\label{thmIntro1Orbi}
\end{theorem}

\begin{proof}
 By Lemma \ref{lemSurCyc}, the projections $p_i(\overline{G})\leq \G_{g_i,\mm_i}$ are either virtually cyclic or of finite index. Thus, after reordering factors, we may assume that there is $s\geq 0$ such that $p_i(\overline{G})\leq \G_{g_i,\mm_i}$ is a finite index subgroup for $1\leq i \leq s$ and virtually cyclic for $s+1\leq i \leq r$. We may further assume that $p_i(\overline{G})=\G_{g_i,\mm_i}$ for $1\leq i \leq s$, since finite index subgroups of orbisurface groups are orbisurface groups. Hence, we obtain a short exact sequence 
 \begin{equation}
 \label{eqnsesCenter}
  1 \rightarrow N \rightarrow \overline{G} \rightarrow p_{1,\dots,s}(\overline{G})\rightarrow 1,
 \end{equation}
where $p_{1,\dots,s}(\overline{G})\leq \G_{g_1,\mm_1}\times \dots \times \G_{g_s,\mm_s}$ is a full subdirect product and $N=\overline{G}\cap (\G_{g_{s+1},\mm_{s+1}}\times \dots \times \G_{g_r,\mm_r})$ is virtually abelian. In particular, $N$ is of type $\mathcal{F}_{\infty}$ and thus the group $p_{1,\dots,s}(\overline{G})$ is of type $\mathcal{F}_m$ \cite[Proposition 2.7]{Bie-81}.

For $1\leq i \leq s$, the kernel of the projection $p_i\circ\phi: G\rightarrow \G_{g_i,\mm_i}$ is the extension

\[
 1 \rightarrow \ker (\phi) \rightarrow \ker(p_i\circ \phi)\rightarrow \overline{G}\cap \left(\G_{g_1,\mm_1}\times \dots \times \G_{g_{i-1},\mm_{i-1}}\times 1 \times\G_{g_{i+1},\mm_{i+1}}\times \dots \times \G_{g_r,\mm_r}\right)\rightarrow 1.
\]

By Theorem \ref{thmBriMil}, the group $$\overline{G}\cap \left(\G_{g_1,\mm_1}\times \dots \times \G_{g_{i-1},\mm_{i-1}}\times 1 \times\G_{g_{i+1},\mm_{i+1}}\times \dots \times \G_{g_r,\mm_r}\right)$$ is finitely generated, since $\overline{G}$ is finitely presented and $\left\{1\right\}\neq\G_{g_i,\mm_i}\cap \overline{G}\unlhd \G_{g_i,\mm_i}$ is normal and thus infinite. Extensions of finitely generated groups by finitely generated groups are finitely generated. Thus, the group $\ker(p_i\circ \phi)$ is finitely generated.

Theorem \ref{lemSiuBeauCat} implies that the epimorphism $p_i\circ \phi: G\rightarrow \G_{g_i,\mm_i}$ is induced by a holomorphic map $f_i: X\rightarrow S_{g_i,\mm_i}$. It follows that the map
\[
f=(f_1,\dots,f_s): X\rightarrow S_{g_1,\mm_1}\times \dots \times S_{g_s,\mm_s}
\] 
is a holomorphic map inducing the compostion $p_{1,\dots,s}\circ \phi$ on fundamental groups.

For any $k\geq 0$ and $1\leq i_1< \dots < i_k\leq s$, the projection $S_{g_1,\mm_1}\times \dots \times S_{g_s,\mm_s}\rightarrow S_{g_{i_1},\mm_{i_1}}\times \dots \times S_{g_{i_k},\mm_{i_k}}$ is holomorphic and hence so is its composition $f_{i_1,\dots,i_k}:X\rightarrow S_{g_{i_1},\mm_{i_1}}\times \dots\times S_{g_{i_k},\mm_{i_k}}$ with $f$. Thus, the homomorphism $p_{i_1,\dots,i_k}\circ \phi: G\rightarrow p_{i_1,\dots,i_k}(\overline{G})$ is induced by the holomorphic map $f_{i_1,\dots,i_k}$.

If $k<2 m$, Corollary \ref{corFinPropsProjFactors} implies that $p_{i_1,\dots,i_k}(\overline{G})\leq \G_{g_{i_1},\mm_{i_1}}\times \dots \times \G_{g_{i_k},\mm_{i_k}}$ has a finite index subgroup of type $\mathcal{F}_m$ and thus is itself of type $\mathcal{F}_m$. Hence, Proposition \ref{propNewCoab} implies that $p_{i_1,\dots,i_k}(\overline{G})$ is virtually holomorphically coabelian.
\end{proof}

We can now complete the proof of Theorem \ref{thmNewCoab}.
\begin{proof}[Proof of Theorem \ref{thmNewCoab}]
The only part that does not follow immediately from Theorem \ref{thmIntro1Orbi} is the furthermore part. Since $\overline{G}\leq \G_{g_1}\times \dots \times \G_{g_r}$ is full and all subgroups of surface groups are either free or surface groups, the virtually cyclic intersections $\overline{G}\cap \G_{g_i}$ are all isomorphic to $\ZZ$. Thus, the group $N$ in the proof of Theorem \ref{thmIntro1Orbi} is isomorphic to $\ZZ^{r-s}$. Since the centralizer of a non-trivial element in a surface group is cyclic we deduce that $\mathrm{Z}(\overline{G})=N=\overline{G}\cap \left(\G_{g_{s+1}}\times \dots \times \G_{g_r}\right) \leq p_{s+1,\dots,r}(\overline{G})\cong \ZZ^{r-s}$ is a finite index subgroup.
\end{proof}

\begin{remark}
Observe that the proof of Theorem \ref{thmIntro1Orbi} shows more generally that if $\phi :G \rightarrow \overline{G}\leq \G_{g_1 ,\mm_1}\times \dots \times \G_{g_r,\mm_r}$ is a homomorphism from $G=\pi_1 (X)$, with $X$ compact K\"ahler, onto a finitely presented full subdirect product $\overline{G}\leq \G_{g_1,\mm_1}\times \dots \times \G_{g_r,\mm_r}$, then there exists a unique complex structure on the product $S_{g_1,\mm_1}\times \dots \times S_{g_r,\mm_r}$ of topological orbifolds such that $\phi = f_{\ast}$ for a holomorphic map $f: X \to S_{g_1,\mm_1}\times \dots \times S_{g_r,\mm_r}$.
\label{rmkNewCoab1}
\end{remark}

Note further that the proof of Theorem \ref{thmIntro1Orbi} also shows that the same conclusion holds if we replace the assumption that $\overline{G}$ is of type $\mathcal{F}_m$ by the assumption that $\overline{G}\leq \G_{g_1,\mm_1}\times \dots \times \G_{g_r,\mm_r}$ is virtually coabelian and finitely presented.
\begin{corollary}
\label{rmkNewCoab}
 Let $G=\pi_1 (X)$ with $X$ compact K\"ahler and let $\phi: G \rightarrow \overline{G}$ be a surjective homomorphisms onto a subgroup $\overline{G}\leq \G_{g_1,\mm_1}\times \dots \times \G_{g_r,\mm_r}$. Assume that $\phi$ has finitely generated kernel and that $\overline{G}$ is virtually coabelian and finitely presented.
 
 Then, for any $0\leq k \leq r$ and any $1\leq i_1 < \dots < i_k\leq r$, the projection $p_{i_1,\dots,i_k}(\overline{G})\leq \G_{g_{i_1},\mm_{i_1}}\times \dots \times \G_{g_{i_k},\mm_{i_k}}$ is virtually holomorphically coabelian. 

\end{corollary}
\begin{proof}
 After passing to a finite index subgroups, we may assume that $\overline{G}$ is coabelian in a direct product of orbisurface groups. Since $\overline{G}$ is finitely presented, it is also a full subdirect product. Now the same arguments as in the proof of Theorem \ref{thmIntro1Orbi} show that $\phi$ is induced by a holomorphic map. The fact that projections of coabelian subgroups of direct products of groups are themselves virtually coabelian completes the proof.
\end{proof}

Corollary \ref{corNewCoabExs} is a direct consequence of Corollary \ref{rmkNewCoab}. More generally, we have 

\begin{corollary}
\label{corNewCoabExs2}
 With the same notation as in Corollary \ref{corNewCoabExs}, for $G=\ker (\psi)$ coabelian of odd rank and $G_1$ any finitely presented group, $H=G \times G_1$ is not K\"ahler. 
\end{corollary}
\begin{proof}
This is an immediate consequence of Corollary \ref{rmkNewCoab} applied to the projection $H\to G$.
\end{proof}

\begin{remark}
 Note that in particular Corollary \ref{corNewCoabExs2} applies to the direct product of any two full subdirect products of orbisurface groups which are coabelian of odd rank. Thus, we can use Corollary \ref{corNewCoabExs2} to construct full subdirect products of orbisurface groups which are coabelian of even rank, but not K\"ahler.
 \label{remNewCoabExs}
\end{remark}

This provides us with large classes of examples of non-K\"ahler subgroups of direct products of surface groups. As we will see in Section \ref{secSESCoab}, many of the examples in Corollary \ref{corNewCoabExs} share the property that they are non-K\"ahler for the reason that their first Betti number is odd.

We want to emphasize that a particularly strong version of Corollary \ref{corNewCoabExs} holds in the case of a direct product of three orbisurface groups.

\begin{theorem}
Let $X$ be a compact K\"ahler manifold and let $G=\pi_1 (X)$. Let $\psi: G \to  \G_{g_1,\mm_1}\times \G_{g_2,\mm_2}\times \G_{g_3,\mm_3}$ be a homomorphism such that the projection $p_i \circ \psi$  has finitely generated kernel for $1\leq i \leq r$ and the image $\overline{G}:= \mathrm{Im}(\psi)$ of $\psi$ is finitely presented. Then one of the following holds:
 \begin{enumerate}
  \item $\overline{G}=\pi_1^{orb} (R)$ for $R$ a closed hyperbolic orbisurface;
  \item $\overline{G}$ is virtually $\ZZ^k$ for $0\leq k\leq 3$
  \item $\overline{G}$ is virtually $\ZZ^k \times \Gamma_h$ for $h\geq 2$ and $k\in \left\{1,2\right\}$;
  \item $\overline{G}$ is virtually a direct product $\ZZ^k\times \Gamma_{h_1}\times \Gamma_{h_2}$ for $h_1, h_2\geq 2$ and $k\in \left\{0,1\right\}$;
  \item $\overline{G}$ is virtually holomorphically coabelian.
 \end{enumerate}
 \label{thm3Factor}
\end{theorem}
\begin{proof}
 Since centralizers in surface groups are cyclic, every free abelian subgroup of $\overline{G}$ has rank $\leq 3$. We will distinguish cases, making repeated use of Lemma \ref{lemSurCyc}. If $\overline{G}$ has virtually cyclic projection to all factors, then $\overline{G}$ is virtually abelian and thus (2) holds. Hence, we may assume that $\overline{G}$ is not virtually abelian. Assume first that $\overline{G}$ is not a full subgroup. After projecting away from factors which have trivial intersection with $\overline{G}$, we either obtain that $\overline{G}$ is a subgroup of a hyperbolic orbisurface group in which case (1) holds, or a full subgroup of a direct product of two hyperbolic orbisurface groups. If the latter holds then either $\overline{G}$ is subdirect in a direct product $\Gamma_{\g_1,\undll_1}\times \Gamma_{\g_2,\undll_2}$ and the VSP property implies that (4) holds, or (3) holds with $k=1$. If $\overline{G}$ is full and not subdirect in any $\G_{\g_1,\undll_1}\times \G_{\g_2,\undll_2}\times \G_{\g_3,\undll_3}$ then we must also have (3) with $k=2$ or (4). Else $\overline{G}$ must be full subdirect (after passing to finite index subgroups of the $\G_{g_i,\mm_i}$) and it follows from finite presentability, Proposition \ref{propNewCoab}, and Remark \ref{rmkNewCoab1} that (5) holds.
\end{proof}

\begin{remark}
Note that the assumption that $p_i\circ \psi$ has finitely generated kernel in Theorem \ref{thm3Factor} can be replaced by the assumption that $p_i \circ \psi$ has either virtually cyclic image or is induced by a holomorphic map for $1\leq i \leq r$, as we can then apply Proposition \ref{propNewCoab} directly to obtain the conclusion. 
\end{remark}

\section{Short exact sequences on abelianizations}
\label{secSESCoab}
Note that for a holomorphic map $f: X\rightarrow Y$, Lemma \ref{lemHodgeHol} does not provide us with the parity of the first Betti number $b_1(f_{\ast}(\pi_1(X)))$ of the group $f_{\ast}(\pi_1 (X))$, but only with the rank of the abelian subgroup $f_{\ast,ab}((\pi_1(X))_{ab})\leq (\pi_1(Y))_{ab}$. This is, because in general the map $f_{\ast,ab}:(\pi_1 (X))_{ab} \rightarrow (\pi_1 (Y))_{ab}$ is not injective. 

We will now show that in the case of coabelian subgroups of direct products of surface groups with strong enough finiteness properties we can actually obtain (virtual) injectivity and as a consequence we obtain that in fact $b_1(f_{\ast}(\pi_1 (X)))$ is even. This will follow from more general group theoretic considerations.

We will make use of the following easy and well-known fact

\begin{lemma}
 Let $G$ and $H$ be groups and let $\phi: G\rightarrow H$ be an injective homomorphism. Then the following are equivalent:
 
 \begin{enumerate}
  \item the induced map $\phi_{ab}: G_{ab}\rightarrow H_{ab}$ on abelianisations is injective;
  \item $\phi(\left[G,G\right])=\phi(G)\cap \left[H,H\right]$.
 \end{enumerate}
 \label{lemInjonAbel}
\end{lemma}

It allows us to prove

\begin{proposition}
 Let $k\geq 1$, $r\geq 2$ be integers, let $G_1,\cdots, G_r$ be finitely generated groups and let $\psi : G=G_1\times \cdots \times G_r \rightarrow \ZZ^k$ be an epimorphism. Assume further that (at least) one of the following two conditions is satisfied:
 \begin{enumerate}
  \item $k=1$ and $\ker (\psi)$ is subdirect in $G_1\times \cdots \times G_r$;
  \item the restriction of $\psi$ to $G_i$ surjects onto $\ZZ^k$ for at least three different $i\in \left\{1,\cdots, r\right\}$.
 \end{enumerate}
 
 Then the map $\psi$ induces a short exact sequence 
 \begin{equation}
  1\rightarrow (\ker (\psi))_{ab}\rightarrow (G_1\times \cdots \times G_r)_{ab}\rightarrow \ZZ^k\rightarrow 1
  \label{eqnSESab1}
 \end{equation}
 on abelianisations and in particular the following equality of first Betti numbers holds:
 \begin{equation}
  b_1(G)=k+ b_1(\ker (\psi)).
  \label{eqnBetti1}
 \end{equation}  
 \label{thmBettiZ}
\end{proposition}

\begin{proof}
 We will first give a proof under the assumption that Condition (2) is satisfied and will then explain how to modify our proof if Condition (1) is satisfied. Assume that Condition (2) holds and that (without loss of generality) the restriction of $\psi$ to each of the first three factors is surjective.
 
 It is clear that exactness of \eqref{eqnSESab1} implies the equality \eqref{eqnBetti1} of Betti numbers. Hence, we only need to prove that the sequence \eqref{eqnSESab1} is exact. Abelianisation is a right exact functor from the category of groups to the category of abelian groups. Hence, it suffices to prove that the inclusion $\iota : \ker (\psi) \rightarrow G$ induces an injection $\iota _{ab} : (\ker (\psi))_{ab}\rightarrow G_{ab}$ of abelian groups.
 
 Since the image of $\psi$ is abelian, it follows that $\left[G,G\right]\leq \ker (\psi)$. We want to show that $\left[G,G\right]\leq \left[\ker(\psi),\ker(\psi)\right]$. Since $\left[G,G\right]=\left[G_1,G_1\right]\times \cdots \times \left[G_r,G_r\right]$ it suffices to show that $\left[G_i,G_i\right]\leq \left[\ker(\psi),\ker(\psi)\right]$ for $1\leq i \leq r$.
 
 We may assume that $i>2$, since for $i=1,2$ the same argument works after exchanging the roles of $G_i$ and $G_3$. Fix $x,y\in G_i$. Since the restrictions $\psi|_{G_j}:G_j\rightarrow \ZZ^k$ are surjective for $j=1,2$ we can choose elements $g_1\in G_1$ and $g_2 \in G_2$ with $\psi(g_1)=-\psi(x)$ and $\psi(g_2)=-\psi(y)$. 
 
 Then the elements $u:=g_1^{-1}\cdot x\in G$ and $v:=g_2^{-1}\cdot y\in G$ are in $\ker (\psi)$. Since $\left[G_i,G_j\right]=\left\{1\right\}$ for $i\neq j$ it follows that $\left[u,v\right]=\left[x,y\right]$. Thus, $\left[G_i,G_i\right]\leq \left[\ker(\psi),\ker(\psi)\right]$ for $1\leq i \leq r$. Consequently $\left[G,G\right]\leq \left[\ker(\psi),\ker(\psi)\right]$ and therefore by Lemma \ref{lemInjonAbel} the map $\iota_{ab}$ is injective.
 
 Now assume that Condition (1) holds. As before it suffices to prove that $\left[G_i,G_i\right]\leq \left[\ker (\psi),\ker(\psi)\right]$ for $1\leq i \leq r$. To simplify notation assume that $i=1$. If we can prove that there is some element $\overline{g}_0\in G_2\times \cdots \times G_r$ such that for any $x\in G_1$ there is an integer $k\in \ZZ$ with $x\cdot \overline{g}_0^k \in \ker (\psi)$ then the same argument as before will show that $\left[G_1,G_1\right]\leq \left[\ker (\psi),\ker(\psi)\right]$.
 
 Observe that we have the following equality of sets
 \begin{align*}
  Q &:=\left\{\psi(g_1,1)\mid (g_1,\overline{g})\in \ker (\psi)\leq G_1\times (G_2\times \cdots \times G_r)\right\}\\ &= \left\{\psi(1,\overline{g})\mid (g_1,\overline{g})\in \ker (\psi) \right\}\leq \ZZ.
 \end{align*}
  
 The set $Q$ is a subgroup of $\ZZ$, since it is the image of the group $\ker (\psi)$ under the homomorphism $\psi\circ \iota_1\circ \pi_1:G\rightarrow \ZZ$, where $\pi_1:G\rightarrow G_1$ is the canonical projection and $\iota_1: G_1 \rightarrow G$ is the canonical inclusion. Let $\overline{g}_0\in G_2\times \cdots \times G_r$ be an element such that $\psi(1,\overline{g}_0)=l_0$ generates $Q$.
 
 Since $\ker (\psi)$ is subdirect, for any $g_1, g_2 \in G_1$ there are elements $\overline{g}_1,\overline{g}_2 \in G_2 \times \cdots \times G_r$ such that $(g_1,\overline{g}_1),(g_2,\overline{g}_2)\in \ker(\psi)$ and therefore $\psi(g_1,1)=k_1\cdot l_0\cdot \psi(g_2,1)=k_2\cdot l_0 \in Q$. It follows that $(g_1,(\overline{g}_0)^{-k_1}), (g_2,(\overline{g}_0)^{-k_2})\in \ker (\psi)$. Thus $\left[g_1,g_2\right] \in \left[\ker (\psi),\ker (\psi)\right]$, completing the proof.
\end{proof}

As a direct consequence, we obtain a constraint on K\"ahler groups

\begin{corollary}
 Let $r,k\geq 1$ be integers, let $G_1,\cdots, G_r$ be finitely generated groups and let $\psi :G_1\times \cdots \times G_r\rightarrow \ZZ^k$ be an epimorphism satisfying one of the Conditions (1) or (2) in Theorem \ref{thmBettiZ}. If $b_1(G_1)+\cdots + b_1(G_r) - k$ is odd then $\ker (\psi)$ is not K\"ahler.
 \label{corBettiZ}
\end{corollary}
\begin{proof}
 By Proposition \ref{thmBettiZ} the first Betti number of $\ker (\psi)$ is equal to $b_1(G_1)+\cdots + b_1(G_r) - k$ and therefore odd. Hence, $\ker (\psi)$ can not be K\"ahler.
\end{proof}

Corollary \ref{corBettiZ} provides us with an elementary proof of Corollary \ref{corNewCoabExs} in two special cases.
\begin{corollary}
 Let $\psi: \G_{g_1,\mm_1}\times \dots \times \G_{g_r,\mm_r}\rightarrow \ZZ$ be a non-trivial homomorphism with subdirect kernel. Then $\ker (\psi)$ has odd first Betti number and in particular is not K\"ahler.
\end{corollary}

\begin{corollary}
 Let $r\geq 3$ and let $\psi: \G_{g_1,\mm_1}\times \dots \times \G_{g_r,\mm_r}\rightarrow \ZZ^{2k+1}$ be a homomorphism such that $\psi|_{\G_{g_i,\mm_i}}$ is an epimorphism for $i=1,2,3$. Then $\ker (\psi)$ has odd first Betti number and in particular is not K\"ahler.
 \label{corOddb1}
\end{corollary}

As a consequence, we obtain:

\begin{corollary}
Let $H=\ker (\psi_1) \times \ker (\psi_2)$ be the product of the kernels of two epimorphisms $\psi_1,\psi_2$ satisfying the conditions in Corollary \ref{corOddb1}. Then $H$ is a full subdirect product of orbisurface groups, which is coabelian of even rank, has even first Betti number and is not K\"ahler.
\label{corExEvenB1}
\end{corollary}
\begin{proof}
 The group $H$ is not K\"ahler by Remark \ref{remNewCoabExs}. All other properties follow from Corollary \ref{corOddb1}.
\end{proof}

\begin{proof}[Proof of Corollary \ref{corExEvenB1Intro}]
This is an easy consequence of Corollary \ref{corExEvenB1}.
\end{proof}

\section{Finiteness properties and first Betti numbers}
\label{secFinPropBetti}

We will now show how the results of Section \ref{secSESCoab} can be used to prove that under the assumption of strong enough finiteness properties the projections in Theorem \ref{thmNewCoab} have a finite index subgroup with even first Betti number.
 
\begin{proposition}
\label{propVB1}
 Let $G=\pi_1 (X)$ and let $\phi: G\rightarrow \G_{g_1,\mm_1}\times \dots \times \G_{g_r,\mm_r}$ be a homomorphism that is induced by a holomorphic map $f: X\rightarrow S_{g_1,\mm_1}\times \dots \times S_{g_r,\mm_r}$ with full image $\overline{G}=\mathrm{im}~ (\phi)$ of type $\mathcal{F}_m$ for $m\geq \frac{2r}{3}$. Then $\overline{G}$ is virtually holomorphically coabelian and there is a finite index subgroup $\overline{G}_0\leq \overline{G}$ with even first Betti number.
\end{proposition}
\begin{proof}
It follows from Proposition \ref{propNewCoab} that $\overline{G}$ is virtually holomorphically coabelian. Let $\G_{h_i}\leq \G_{g_i,\mm_i}$ be finite index surface subgroups, $1\leq i \leq r$ and $\phi: \G_{h_1}\times \dots \times \G_{h_r}\rightarrow \ZZ^{2N}=:Q$ be an epimorphism such that $\ker (\phi) \leq \overline{G}$ is a finite index subgroup. 

By Lemma \ref{lemSurjComp}, the group $\phi\left(p_{i_1,\dots, i_{r-m}}(\overline{G})\right)  \leq Q$ is a finite index subgroup for all $1\leq i_1 < \dots < i_{r-m}\leq r$. Since $m\geq \frac{2r}{3}$, there is a partition of $\left\{1,\dots, r\right\}$ into three disjoint subsets $B_1=\left\{j_0=1,\dots, j_1\right\}$, $B_2=\left\{j_1+1,\dots, j_2\right\}$ and $B_3= \left\{j_2+1,\dots, j_3=r\right\}$ such that $|B_i|\geq r-m$. Thus, $A_i:= \phi\left(p_{B_i}(\overline{G})\right)\leq Q$ is a finite index subgroup. 

Define $P_i = \G_{h_{j_{i-1}}}\times \dots \times \G_{h_{j_i}}$. Consider the finite index subgroup $A=A_1\cap A_2 \cap A_3\leq \ZZ^{2l}$ and define finite index subgrups $P_{i,0}=\phi^{-1}(A)\cap P_i \leq P_i$. Consider the restriction $\overline{\phi}: P=P_{1,0}\times P_{2,0}\times P_{3,0}\rightarrow A$ of $\phi$ to the finite index subgroup $P\leq \G_{h_1}\times \dots \times \G_{h_r}$. 

Since $\phi(P_i)=A_i\geq A$, we have $\overline{\phi}(P_{i,0})=A$. Thus, the projection of $\overline{G}_0 = \ker (\overline{\phi})$ onto the $P_{i,0}$ is surjective and we can therefore apply Proposition \ref{thmBettiZ}. Hence, the induced homomorphism $\left(\ker (\overline{\phi})\right)_{ab}= (\overline{G}_0)_{ab}\rightarrow P_{ab}$ is injective and
\[
 b_1(\overline{G}_0) = b_1(P)- 2N= b_1(P_{1,0})+b_1(P_{2,0})+b_1(P_{3,0})-2N.
\]
However, $b_1(P_{i,0})$ is even, because $P_{i,0}$ is a finite index subgroup of the K\"ahler group $P_i$ for $i=1,2,3$. Thus, we obtain
\[
b_1(\overline{G}_0)\equiv 0 ~ \mathrm{ mod } ~ 2
\]
for the finite index subgroup $\overline{G}_0\leq \overline{G}$.
\end{proof}

\begin{corollary}
 With the notation of Theorem \ref{thmIntro1Orbi}, assume that the group $\overline{G}$ has finiteness type $\mathcal{F}_m$. Then the projections onto $k\leq \frac{3m}{2}$ factors with indices $1\leq i_1 < \dots < i_k\leq s$ have a finite index subgroup with even first Betti number. 
 
 If, moreover, $G=\overline{G}$ is a subgroup of a direct product of surface groups and $s\leq \frac{3m}{2}$, then $r-s$ is even.
\end{corollary}
\begin{proof}
The first part is an immediate consequence of Proposition \ref{propVB1} and the fact that the homomorphism in Theorem \ref{thmIntro1Orbi} is virtually induced by a holomorphic map.

For the second part, observe that this puts us in the setting of Theorem \ref{thmNewCoab}. Thus, after possibly passing to the finite index subgroup $G\cap \left(\G_{g_1}\times \dots \times \G_{g_s} \times \mathrm{Z}(G)\right)\leq G$, we may assume that $G$ is a direct product $G\cong G_0 \times \ZZ^{r-s}$ with $G_0\leq \G_{g_1}\times \dots \times \G_{g_s}$ full subdirect. Since $s\leq \frac{3m}{2}$, $G_0$ has a finite index subgroup $G_1\leq G_0$ with even first Betti number. The group $G_1 \times \ZZ^{r-s}$ is K\"ahler, so it also has even first Betti number. Therefore, $r-s$ is even.
\end{proof}

\section{Universal homomorphism}
\label{secConsGens}

Delzant \cite[Theorem 2]{Del-08} and Corlette--Simpson \cite[Proposition 2.8]{CorSim-08} proved that for a K\"ahler group $G$ there is a finite number of Riemann orbisurfaces $S_{g_i,\mm_i}$ such that every epimorphism from G onto a surface group factors through one of the $\pi_1 ^{orb} (S_{g_i,\mm_i})$. It is not difficult to see that their result can be stated as follows:

\begin{theorem}
 Let $X$ be a compact K\"ahler manifold and let $G=\pi_1 (X)$ be its fundamental group. Then there is $r\geq 0$ and closed hyperbolic Riemann orbisurfaces $S_{g_i,\mm_i}$ of genus $g_i \geq 2$ together with surjective holomorphic maps $f_i : X \rightarrow S_{g_i,\mm_i}$ with connected fibres, $1\leq i \leq r$, such that
 \begin{enumerate}
  \item the induced homomorphisms $f_{i,\ast} : G \rightarrow \pi_1 ^{orb} (S_{g_i,\mm_i})$ are surjective with finitely generated kernel for $1\leq i \leq r$;
  \item the image of $\phi:= (f_{1,\ast}, \dots, f_{r,\ast}): G \rightarrow  \pi_1^{orb} (S_{g_1, \mm_1}) \times \dots \times \pi_1^{orb} (S_{g_r,\mm_r})$ is full subdirect; and
  \item every epimorphism $\psi: G\rightarrow \pi_1^{orb} (S_{h,\nn})$ onto a fundamental group of a closed hyperbolic Riemann orbisurface of genus $h\geq 2$ and multiplicities $\nn$ factors through $\phi$.
 \end{enumerate}
 \label{thmDelCorSim}
\end{theorem}

For a K\"ahler group $G$ we will call the homomorphism $\phi$ in Theorem \ref{thmDelCorSim} the \textit{universal homomorphism} (to a product of hyperbolic Riemann orbisurfaces), as it satisfies a universal property. 

\begin{lemma}
 Let $X$ be a compact K\"ahler manifold, let $G=\pi_1(X)$ be its fundamental group and let $\phi: G \rightarrow \pi_1^{orb} (S_{g_1,\mm_1}) \times \dots \times \pi_1 ^{orb} (S_{g_r,\mm_r})$ be the universal homomorphism to a product of orbisurface groups defined in Theorem \ref{thmDelCorSim}. 
 
 Then $\phi$ is induced by a holomorphic map $f: X\rightarrow S_{g_1,\mm_1} \times \dots \times S_{g_r,\mm_r}$ and the image $G:=\phi(G)\leq \pi_1^{orb} (S_{g_1,\mm_1}) \times \dots \times \pi_1^{orb} (S_{g_r,\mm_r})$ of $\phi$ is a finitely presented full subdirect product. 
 \label{lemUnivFinPres}
\end{lemma}
\begin{proof}
To simplify notation denote by $\G_i:= \pi_1^{orb} (S_{g_i,\mm_i})$ the orbifold fundamental group of $S_{g_i,\mm_i}$ for $1\leq i \leq r$. The only part that is not immediate from Theorem \ref{thmDelCorSim} is that the image $\overline{G}$ of the restriction $\phi|_{G}$ is finitely presented. 

To see this, recall that by Theorem \ref{thmDelCorSim} the composition $p_i\circ \phi: G\rightarrow \pi_1^{orb} (S_{g_i,\mm_i})$ of $\phi$ with the projection $p_i$ onto $\G_i$ has finitely generated kernel. Hence, the kernel 
 \[
 N_i:=\ker (p_i|_{\overline{G}}) = \overline{G} \cap \left( \G_1 \times \dots \times \G_{i-1}\times 1 \times \G_{i+1} \times \dots \times \G_r\right)\unlhd \overline{G}
 \]
of the surjective restriction $p_i|_{\overline{G}} : \overline{G} \rightarrow \G_i$ is a finitely generated normal subgroup of $\overline{G}$. 

Finite presentability is trivial for $r=1$, so assume that $r\geq 2$. Let $1\leq i < j \leq r$. The image of the projection $p_{i,j}(\overline{G}) \leq \G_i \times \G_j$ is a full subdirect product and $p_{i,j}(N_i)\unlhd p_{i,j}(\overline{G})$ is a normal finitely generated subgroup. Since by definition $p_{i,j}(N_i)\leq 1\times \G_j$, it follows from subdirectness of $p_{i,j}(\overline{G})$ that in fact $p_{i,j}(N_i)\unlhd 1 \times \G_j$ is a normal finitely generated subgroup. 

The group $p_{i,j}(N_i)$ is either trivial or has finite index in $\G_j$, since all finitely generated normal subgroups of $\G_j=\pi_1^{orb} (S_{g_j,\mm_j})$ are either trivial or of finite index. The former is not possible, because $\overline{G}$ is full. It follows that $p_{i,j}(N_i)\unlhd 1 \times \G_j$ is a finite index subgroup. Thus, $p_{i,j}(\overline{G}) \leq \G_i \times \G_j$ is a finite index subgroup. Since $i$ and $j$ were arbitrary, we obtain that $\overline{G}$ has the VSP property. Thus, $\overline{G}$ is finitely presented by \cite[Theorem A]{BriHowMilSho-13}.
\end{proof}

Versions of the results of Sections \ref{secResCoabKGs} to \ref{secFinPropBetti} hold for the universal homomorphism of a K\"ahler group. This is because the universal homomorphism is induced by a holomorphic map. In particular, we obtain the following version of Theorem \ref{thmIntro1Orbi}.

\begin{theorem}
 For every K\"ahler group $G$ there are $r\geq 0$, closed orientable hyperbolic orbisurfaces $S_{g_i,\mm_i}$ of genus $g_i\geq 2$ and a homomorphism $\phi: G \rightarrow \pi_1^{orb} (S_{g_1,\mm_1})\times \dots \times \pi_1^{orb} (S_{g_r,\mm_r})$ with the universal properties described in Theorem \ref{thmDelCorSim}. Its image $\overline{G}=\phi(G)\leq \pi_1^{orb} (S_{g_1,\mm_1})\times \dots \times \pi_1^{orb} (S_{g_r,\mm_r})$ is a finitely presented full subdirect product. 
 
 Let $k\geq 0$ and $m\geq 2$ such that $m > \frac{k}{2}$. If $\overline{G}$ is of type $\mathcal{F}_m$ then, for every $1\leq i_1 < \dots < i_k \leq r$, the projection $p_{i_1,\dots,i_k}(\overline{G})\leq \pi_1^{orb} (S_{g_{i_1},\mm_{i_1}})\times \dots \times \pi_1^{orb} (S_{g_{i_k},\mm_{i_k}})$ has a finite index subgroup which is virtually holomorphically coabelian.
 \label{thmExistRestFinPropsMT}
\end{theorem}

\begin{proof}
  The assertion that $\overline{G}$ is finitely presented follows from Lemma \ref{lemUnivFinPres}. By Theorem \ref{thmDelCorSim} $\phi$ is induced by a holomorphic map. We can lift any such holomorphic map to a holomorphic map $f$ defining the restriction  $\phi|_{G_0}: G_0 \rightarrow \pi_1 (R_{\g_1})\times \dots \times \pi_1 (R_{\g_r})$, obtained by passing to finite index surface subgroups $\pi_1 (R_{\g_i})\leq \pi_1^{orb} (S_{g_i,\mm_i})$ and the finite index subgroup $G_0:= G \cap \phi^{-1}(\pi_1 (R_{\g_1})\times \dots \times \pi_1 (R_{\g_r}))$. The result now follows from Corollary \ref{corFinPropsProjFactors} and Proposition \ref{propNewCoab}.   
\end{proof}

\begin{remark}
 Note that since $\overline{G}$ is finitely presented, the consequences of Theorem \ref{thmExistRestFinPropsMT} always apply for $k=3$.
\end{remark}

Lemma \ref{lemUnivFinPres} and its proof raise the natural question if there is a geometric analogue of the VSP property. The following result shows that this is indeed the case.

\begin{proposition}
 Let $X$ be a compact K\"ahler manifold and let $G=\pi_1 (X)$. Let $\phi: G \to  \pi_1^{orb} (S_{g_1,\mm_1}) \times \cdots \times \pi_1^{orb} (S_{g_r,\mm_r})$ be a homomorphism with finitely presented full subdirect image such that the composition $p_i\circ \phi : G \to \pi_1^{orb} (S_{g_i,\mm_i})$ has finitely generated kernel for $1\leq i \leq r$.
 
 Then $\phi=f_{\ast}$ is realised by a holomorphic map $f=(f_1,\cdots,f_r): X \rightarrow S_{g_1,\mm_1} \times \cdots \times S_{g_r,\mm_r}$, for suitable complex structures on the $S_{g_i,\mm_i}$, and the holomorphic projection $f_{ij}=(f_i,f_j): X \rightarrow S_{g_i,\mm_i} \times S_{g_j,\mm_j}$ is surjective for $1\leq i < j \leq r$.
\label{propUnivFinPresHol}
\end{proposition}
\begin{proof}
 Finite generation of $\ker(p_i\circ \phi)$ and Theorem \ref{lemSiuBeauCat} imply that $\phi$ is induced by a holomorphic map $f= (f_1,\cdots,f_r): X \rightarrow S_{g_1,\mm_1} \times \cdots \times S_{g_r,\mm_r}$. We now replace the $S_{g_i,\mm_i}$ by closed hyperbolic Riemann surfaces $R_i$, by passing to regular finite covers $R_{i}\rightarrow S_{g_i,\mm_i}$ and the induced finite-sheeted cover $X_0\rightarrow X$ satisfying that $\pi_1 (X_0) = f_{\ast}^{-1}\left((\pi_1 (R_{1})\times \dots \times \pi_1 (R_{r}) ) \cap f_{\ast}(\pi_1 (X))\right)$. Since the kernels of the epimorphisms $\pi_1 (X_0)\to \pi_1 (R_{\g_i})$ are finitely generated, the fibres of the induced holomorphic maps $h_i: X_0\to R_{\g_i}$ are connected.
 
Assume for a contradiction that there is $1\leq i < j \leq r$ such that the image $h_{ij}(X_0)\subset R_{i}\times R_{j}$ of the holomorphic map $h_{ij}=(h_i,h_j): X_0 \rightarrow R_{i}\times R_{j}$ is 1-dimensional.
  
Let $D_i$ (resp. $D_j$) be the finite set of singular values of $h_i$ (resp. $h_j$) and let $D= h_{ij}(X_0) \cap \left(D_i\times R_{j} \cup R_{i} \times D_j\right)$. Since $h_i$ and $h_j$ are both surjective, the set $D$ is finite and by definition all points in $h_{ij}(X_0)\setminus D$ are of the form $(x_i,x_j)$ with $x_i$ and $x_j$ regular values of $h_i$ and $h_j$. Let $(x_i,x_j)\in h_{ij}(X_0)\setminus D$. The map $h_{ij}$ has rank 1 in each regular point, implying that $\mathrm{d}h_i(y)$ and $\mathrm{d}h_j(y)$ are proportional for every point $y\in h_{ij}^{-1}(x_i,x_j)$. Since the fibres of $h_i$ and $h_j$ are connected it follows that $h_i^{-1}(x_i)=h_j^{-1}(x_j)$. We deduce that for $F=p_i(D)$ we have a holomorphic map $q: R_{i}\setminus F\to R_{j}$ such that the diagram 
\[
\xymatrix{ X_0 \setminus h_i^{-1}(F)\ar[d]_{h_i}\ar[rd]^{h_j}&\\ R_{i}\setminus F\ar[r]& R_{j}}
\]
commutes. Since $F$ is finite $q$ extends to a holomorphic map $q: R_{i}\to R_{j}$, which is biholomorphic by symmetry of the argument. In particular, the image $h_{ij,\ast}(\pi_1 (X_0))\leq \pi_1(R_{i})\times \pi_1(R_{j})$ is isomorphic to $\pi_1(R_{i})\cong\pi_1(R_{j})$. 

In contrast the VSP property and the fact that $\overline{G}$ is a finitely presented full subdirect product imply that $h_{ij,\ast}(\pi_1(X_0))\leq \pi_1(R_{i})\times \pi_1 (R_{j})$ is a finite index subgroup. This is a contradiction. It follows that $h_{ij}(X_0)$ is 2-dimensional for $1\leq i<j\leq r$.
\end{proof}
In a previous version we proved Proposition \ref{propUnivFinPresHol} using Stein factorization. We are grateful to the referee for providing us with the simpler proof given above.

\begin{remark}
We have seen in the proof of Theorem \ref{thmIntro1Orbi} in Section \ref{secResCoabKGs} that the kernel of $\phi$ in Proposition \ref{propUnivFinPresHol} being finitely generated is a sufficient condition for $\phi$ to be induced by a holomorphic map, and have the property that the projections of $G$ to surface group factors have finitely generated kernel.
\end{remark}

Note that the consequences of Proposition \ref{propUnivFinPresHol} in particular apply to the universal homomorphism to a product of orbisurfaces.
\begin{corollary}
Let $X$ be a compact K\"ahler manifold and let $f=(f_1,\cdots,f_r): X \rightarrow S_{g_1,\mm_1} \times \cdots \times S_{g_r,\mm_r}$  be a holomorphic realisation of the universal homomorphism $\phi: G\rightarrow \pi_1^{orb} (S_{g_1,\mm_1}) \times \cdots \times \pi_1^{orb} (S_{g_r,\mm_r})$ defined in Theorem \ref{thmDelCorSim}.

  Then the holomorphic projection $f_{ij}=(f_i,f_j): X \rightarrow S_{g_i,\mm_i} \times S_{g_j,\mm_j}$ is surjective for $1\leq i < j \leq r$.
\end{corollary}
\begin{proof}
This is an immediate consequence of applying Proposition \ref{propUnivFinPresHol} to Lemma \ref{lemUnivFinPres} and its proof.
\end{proof}

It is natural to ask if there is a generalisation of Proposition \ref{propUnivFinPresHol} to give surjective holomorphic maps onto products of $s$ factors. The examples constructed in Theorem \ref{thmExsGenClass} show that this is certainly false for general $r-1 \geq s\geq 3$ -- for instance consider Theorem \ref{thmExsGenClass} with $r=4$, $k=2$ and any choice of branched coverings satisfying all necessary conditions. More generally, we also note that all of the groups constructed in Theorem \ref{thmExsGenClass} are projective. Thus, by the Lefschetz Hyperplane Theorem, they can be realised as fundamental groups of compact projective surfaces. Hence, we can not even hope for holomorphic surjections onto $k$-tuples under the additional assumption that our groups are of type $\mathcal{F}_m$ and that $k\leq m$. This shows that any result regarding geometric surjection to $k$-tuples would necessarily have to be of a more subtle nature.

\bibliography{References}
\bibliographystyle{amsplain}

\end{document}